\documentclass[11pt]{amsart}
\usepackage{amssymb,graphicx,hyperref, mathtools}
\usepackage{amsmath, amscd, amsthm, amssymb, graphics, mathrsfs, setspace,fancyhdr,tabularx,shapepar, hyperref, xcolor, tikz, cite, booktabs, dsfont,amsfonts, marvosym, amsbsy, bm, tikz, array,MnSymbol}
\usepackage[utf8]{inputenc}
\usepackage{braket}
\usepackage{tikz}
\usepackage{tikz-cd}
\usetikzlibrary{arrows, matrix}
 
 \usepackage[margin=1in]{geometry}

\def\to{\mathchoice{\longrightarrow}{\rightarrow}{\rightarrow}{\rightarrow}}

\newcommand{\ie}{i.e.,\ }
\newcommand{\ra}{\rightarrow}
\newcommand{\dg}{\operatorname{deg}}

\newcommand{\sop}{\; { }_s\!\otimes_p}
\newcommand{\X}{\mathbf X}
\newcommand{\sopl}{\; { }_s\!\otimes^L_p}
\newcommand{\kcr}{\operatorname{CR}^{\mathbb{G}_m}_k}
\newcommand{\Spec}{\operatorname{Spec}}
\newcommand{\QCoh}{\operatorname{QCoh}}

\newcommand{\Perf}{\operatorname{Perf}}

\DeclareMathOperator{\Hom}{Hom}

\newcommand{\sslash}[1]{/ \! \! / \!_{#1}}
\theoremstyle{plain}
\newtheorem{thm}{Theorem}[section]
\newtheorem{prop}[thm]{Proposition}
\newtheorem{cor}[thm]{Corollary}
\newtheorem{lem}[thm]{Lemma}

\theoremstyle{definition}
\newtheorem{defin}[thm]{Definition}
\newtheorem{rem}[thm]{Remark}
\newtheorem{exa}[thm]{Example}

\numberwithin{equation}{section}

\begin{document}
	
	\title[Windows for cdgas]{Windows for cdgas}
	\author[Chidambaram]{Nitin K. Chidambaram}
	\address{
		\begin{tabular}{l}
			Nitin K. Chidambaram \\
			\hspace{.1in} Max-Planck-Institut f\"ur Mathematik\\
			\hspace{.1in} Vivatsgasse 7, 53111 Bonn, Germany \\
			\hspace{.1in} Email: {\bf kcnitin@mpim-bonn.mpg.de} \\
			\hspace{.1in} Webpage: \url{https://guests.mpim-bonn.mpg.de/kcnitin/} \\
		\end{tabular}
	}
\author[Favero]{David Favero}
\address{
	\begin{tabular}{l}
		David Favero \\
		\hspace{.1in} University of Alberta \\
		\hspace{.1in} Department of Mathematical and Statistical Sciences \\
		\hspace{.1in} Central Academic Building 632, Edmonton, AB, Canada T6G 2C7 \\
		\\
		\hspace{.1in} Korea Institute for Advanced Study \\
		\hspace{.1in} 85 Hoegiro, Dongdaemun-gu, Seoul, Republic of Korea 02455 \\
		\hspace{.1in} Email: {\bf favero@ualberta.ca} \\
		\hspace{.1in} Webpage: \url{https://sites.ualberta.ca/~favero/} \\
	\end{tabular}
}

\begin{abstract}
We study a Fourier-Mukai kernel associated to a GIT wall-crossing for arbitrarily singular (not necessarily reduced or irreducible) affine varieties over any field.  This kernel is closely related to a derived fiber product diagram for the wall-crossing and simple to understand from the viewpoint of commutative differential graded algebras.  However, from the perspective of algebraic varieties, the kernel can be quite complicated, corresponding to a complex with multiple homology sheaves.  Under mild assumptions in the Calabi-Yau case, we prove that this kernel provides an equivalence between the category of perfect complexes on the two GIT quotients.   More generally, we obtain semi-orthogonal decompositions which show that these categories differ by a certain number of copies of the derived category of the derived fixed locus.   The derived equivalence for the Mukai flop is recovered as a very special case.

\end{abstract}

\maketitle

\tableofcontents

\section{Introduction}

Flops are some of the most elementary birational transformations. A deep relationship between birational geometry and derived categories originated in the work of  Bondal and Orlov \cite{BO}.  They conjectured that flops have equivalent derived categories; Kawamata refined and generalized this further in his famous K-equivalence implies D-equivalence conjecture\cite{Kaw}.

There are various  solutions to this problem for explicit types of flops; for example, Bondal and Orlov \cite{BO} for the standard/Atiyah flop, Bridgeland\cite{BridgeFlops} for flops in dimension $ 3 $, and Namikawa and Kawamata for (stratified) Mukai flops \cite{NamiFlops, KawFlops, Cautis}. However there is no agreed upon solution to tackle this problem in full generality.

One of the various successful techniques in addressing this problem, known as ``grade-restriction windows'', was introduced by Segal\cite{Seg} based on ideas from string theory revealed by physicists Herbst, Hori, and Page\cite{HHP}.  Segal's work and the introduction of this technique quickly lead to significant generalizations by Ballard, Favero and Katzarkov \cite{BFKe} and Halpern-Leistner\cite{HL} i.e.\ grade-restriction windows can be used to prove Kawamata's D-equivalence implies K-equivalence conjecture for certain Variation of Geometric Invariant Theory (VGIT) problems, even in non-abelian cases.

Due to a result of Reid (see\cite[Proposition 1.7]{Tha}), one can reduce the problem of flops to a VGIT problem. In more detail, given a flop 
\[\begin{tikzcd}
Y_1 \arrow[r,dotted] & Y_2,
\end{tikzcd}
\] one can find a scheme  $ Y $ with a $ \mathbb{G}_m $-action, such that $ Y_1 $ and $ Y_2 $ are realized as two different GIT quotients.  Hence, we may try to prove Bondal-Orlov--Kawamata's conjecture in general using the variation of GIT quotients.
In the known cases, the idea is as follows.

Consider a scheme $ Y $ equipped with an action of a linearly reductive group $ G $. We denote the GIT quotient with respect to  an $ G $-equivariant ample line bundle as $ Y\sslash{G} $. The \textit{window functor} is a fully faithful functor 
\[
\Phi : D(Y\sslash{G}) \longrightarrow D(\operatorname{QCoh}^G Y).
\] The essential image of this functor is referred to as the \textit{window subcategory} or as the window in short. By choosing a different linearization, we can consider a different GIT quotient and construct a different window functor. The  proof of the equivalence between the derived categories of the different GIT quotients works by comparing the two window subcategories in $  D(\operatorname{QCoh}^G Y) $.

Though it is not constructed this way in \cite{BFKe, HL}, the window functor $\Phi$ is expected to be of geometric origin.  Namely, it should be defined by a Fourier-Mukai kernel $ P $ which is an object in $ D(Y\sslash{G} \times [Y/{G}]) $, such that
\[
\Phi(F) = {\pi_2}_*(P \otimes \pi_1^* F),
\] where $ \pi_1 $ and $ \pi_2 $ are projections from $ Y\sslash{G} \times [Y/{G}] $ to the first and second factor respectively.  We note here that even in examples of flops where the derived equivalences are known to hold, there is not always a construction of the Fourier-Mukai kernel that induces this equivalence.

Recently, Ballard, Diemer, and Favero \cite{bdf}  proposed a consistent way to produce homologically  well behaved  Fourier-Mukai kernels for these window functors in the case of $ \mathbb{G}_m  $-actions, called the \textit{$ Q$-construction}. The definition of the object $ Q $ has the following geometric motivation: it is the  partial compactification of the group action, and when the group is $ \mathbb G_m $, it generalizes a construction of Drinfeld \cite{drin}. The object $ Q $ satisfies various nice properties; it is a functorial assignment (see \cite[Section 2]{bdf} for details), and it behaves well with respect to base change. Generalizing this definition of $ Q $ to actions of more general groups (i.e.\  not $ \mathbb G_m $) is not obvious, but an extension to $ GL_n $-actions for Grassmannian flops was analyzed  in \cite{BCFMV}.  

In \cite{bdf}, the authors recovered grade-restriction windows for \textit{smooth} affine schemes $ X $ with a $ \mathbb{G}_m $-action via an explicit construction of a kernel. However, birational geometry demands that we study more than smooth schemes.  Indeed, even in the setting of smooth flops, the scheme $ X $ is often singular (for example, the Mukai flop). In the case of singular affine schemes Ballard, Diemer, and Favero proposed a \textit{derived $ Q $-construction}, which uses techniques from derived algebraic geometry.   Furthermore, they conjecture that this derived $ Q $-construction provides a Fourier-Mukai equivalence for flops (i.e.\ they propose an explicit method to prove Bondal-Orlov-Kawamata's conjecture). 

In order to tackle this problem of VGIT for singular schemes, the idea is to resolve the singular scheme, by a smooth dg-scheme and then apply the $  Q $-construction. Instead of working in the simplicial setting (as \cite{bdf} does), we use the philosophy of the monoidal Dold-Kan correspondence to work in the dg setting instead. 

\subsection{Windows}

In this paper, we develop a theory of variation of GIT quotients arising from semi-free commutative differential graded algebras (cdga). This uses the language of dg-schemes as developed in \cite{CFK, Riche, BR, MR}. We prefer to use this over the more involved machinery of derived algebraic geometry (DAG) as it is  explicit and better suited (at least for us) for computations.

Let us start with a $ \mathbb Z $-graded semi-free (\ie free upon forgetting the differential) cdga $ R $ over a smooth finitely-generated $ k $-algebra $ T $, where $ k $ is an arbitrary field. (We can allow $ k $ to be an arbitrary Noetherian ring for  some of the following results, but we restrict to the case of field in the introduction for simplicity.) We will  assume that the homological degree zero part is $ R^0 = T $. This data gives a dg-scheme with a $\mathbb G_m$-action (for more details about dg-schemes, see Section~\ref{sec:dg}),
	\[
	\X := (X, \mathcal R).
	\] Here, $ X = \Spec T $, and $ \mathcal R $ is the cdga $ R $ considered as a dg-$ \mathcal O_X $-module on $ X $. Let us denote the ideal generated by all the positively graded elements in $ T $ by  $ J^+ $, and the associated semi-stable locus as  
	\[
	X^+ := X - V(J^+).
	\] We can also consider the restriction of the sheaf of cdgas $ \mathcal R $ to the semi-stable locus, and denote it by $ \mathcal R^+$. Then, we define the positive semi-stable locus of the dg-scheme $ \X $ as 
	\[
	\X^+ := (X^+, \mathcal R^+).
	\]
	
In Section~\ref{sec:Qcdga}, we define a Fourier-Mukai kernel $Q_+$ and following \cite{bdf}, we  get a fully-faithful functor $ \Phi_{Q_+} $ without any assumptions.
\begin{prop}[Lemmas~\ref{lem:FF} and~\ref{lem:QP}]\label{prop:ffwf}
There is a fully-faithful (window) functor
\[
\Phi_{Q_+} : D(\QCoh^{\mathbb{G}_m}(\X^+)) \longrightarrow D(\QCoh^{\mathbb{G}_m}(\X)).
\]
\end{prop}
In order to identify the essential image of the window functor (known as a window) explicitly, we need to impose a condition on the internal degree of the homological generators of $ R $ over $ T $. We  also have to restrict to the subcategory of perfect objects (as in Definition~\ref{def:perf}).
\begin{thm}[Theorem \ref{thm:Qequiv}]\label{thm:introQequiv}
Assume that $R$ is generated as a $T$-algebra by non-positive elements. The functor
	\[
	\Phi_{Q_+}  : \operatorname{Perf}^{\mathbb G_m}(\X^+) \longrightarrow \mathbb{W}^+_{\X}
	\] is an equivalence of categories where $\mathbb{W}^+_{\X}$ is \'etale locally the full subcategory generated by $\mathcal R(i)$ for $i$ in a prescribed range (see Definition~\ref{defin: window}). 
\end{thm}

Perhaps a more invariant way of  requiring that $ R $ is generated by (internal) degree zero elements over $ T $, is to require that the restriction  of the relative cotangent complex  $	\mathbb L_{R/T} $  to the fixed locus lives in (internal) degree zero. See Remark~\ref{rem:Ldeg} for more details.

\subsection{Wall-crossing for cdgas}
Using the window functor and the explicit description of the window from the previous section, we can compare the derived categories of the positive and negative GIT quotients. Thus, we obtain the following results about wall-crossings.

 By considering  the ideal $ J^- $ generated by the negatively graded elements of $ T $, we may define the negative semi-stable locus  of the dg-scheme $ \X $ analogously to the previous section, \ie
	\[
	\X^- := (X^-, \mathcal R^-).
	\] We can also prove the analogue of the window statement (Theorem~\ref{thm:introQequiv}) for $ \X^- $ if we assume that the cdga $ R $ is generated by non-negatively graded elements. We can combine these results to understand the wall-crossings between $ \X^+ $ and $ \X^- $.
	
	Let $ \mu_{\pm} $ be the sum of the weights of the conormal bundle of $ \operatorname{Spec} T/J^{\pm} $ in $ X $. Let $ j_{\pm} : \X^{\pm} \hookrightarrow \X $ be the inclusion of dg-schemes. 
	\begin{thm}[Theorem \ref{thm:wcequiv}  ]\label{thm:intro1} Assume that the dg-algebra generators of $R$ over $T$ have internal degree zero. When, $ \mu_{+} + \mu_{-} = 0 $, the wall crossing functor
			\[
			\Phi^{\operatorname{wc}}_+ := j_{-}^* \circ \left(-\otimes\mathcal{O}(-\mu_+-1)\right)\circ \Phi_{Q_+} : \operatorname{Perf}^{\mathbb{G}_m}(\X^+) \longrightarrow \operatorname{Perf}^{\mathbb{G}_m}(\X^-)
			\] is an equivalence of categories. The inverse functor is 
			\[
			\Phi^{\operatorname{wc}}_- := j_{+}^* \circ \left(-\otimes\mathcal{O}(-\mu_-+1)\right)\circ \Phi_{Q_-} 
			\]
	\end{thm}

We also study the case where $ \mu_{+} + \mu_- \neq 0 $ and $ X $ is isomorphic to an affine space over an arbitrary Noetherian ring $ k $.
 Consider the ring $$ T = k[\mathbf x, \mathbf y], $$ where we use the shorthand notation $ \mathbf x  $ to mean $ x_1,x_2,\cdots,x_l $ and $   \mathbf y   $ to mean $ y_1,y_2,\cdots,y_m $ with internal $ \mathbb Z $-grading $ \dg x_i  > 0 $ and $ \dg y_i  < 0  $.

 In this setting, we have the following result.

\begin{thm}[Theorem \ref{thm:wcsod}] \label{thm:intro2} Let $ k $ be an arbitrary Noetherian ring. Assume that the algebra generators of $R$ over $ T = k[\mathbf x, \mathbf y]$ have internal degree zero. Consider the (derived) fixed locus $ R^{\mathbb G_m} := R/(\mathbf x, \mathbf y)$.	
	\begin{enumerate}
	\item When, $ \mu_{+} + \mu_{-} > 0 $,  we have the following semi-orthogonal decomposition
	\[
	\Perf^{\mathbb G_m}\X^+  \cong \langle \Perf(R^{\mathbb G_m})_{\mu_++\mu_-},\cdots,  \Perf(R^{\mathbb G_m})_2,  \Perf(R^{\mathbb G_m})_1, \Phi^{wc}_-( \Perf^{\mathbb G_m}\X^-)\rangle.
	\] 
	\item When, $ \mu_{+} + \mu_{-} < 0 $,  we have the following semi-orthogonal decomposition
	\[
	\Perf^{\mathbb G_m}\X^-  \cong \langle  \Perf(R^{\mathbb G_m})_{\mu_+ + \mu_-},\cdots, \Perf(R^{\mathbb G_m})_{-2},  \Perf(R^{\mathbb G_m})_{-1}, \Phi^{wc}_+( \Perf^{\mathbb G_m}\X^+)\rangle.
	\] 
\end{enumerate}
\end{thm}

The main application of the above results that we are interested in is the setting of VGIT problems arising from singular schemes (in particular, those obtained as a VGIT presentation of a flop). 

\subsection{Applications to singular VGIT problems}\label{sec:introVGIT}

Consider  \emph{any} closed subscheme $ Y = \Spec S $ of  a smoooth affine scheme $ X = \Spec T $ of finite type over a field $ k $, which is equipped with a $ \mathbb G_m $-action and equivariant embedding,
\[
Y = \Spec S \hookrightarrow X = \Spec T.
\] 

Then we resolve the $ \mathbb Z $-graded singular ring $ S $ by a $ \mathbb Z $-graded semi-free cdga $ R $, using the Koszul-Tate resolution (which is a generalization of the Koszul resolution for complete intersections) \cite{Tate},

\[
R \simeq S.
\] In particular, $ R $ has only finitely many generators in each homological degree, but could have infinitely many generators in total. Note that the degree zero part $ R^0 = T $. We introduce the following notation for the GIT quotients of $ Y $.
\[
Y\sslash{\pm} :=  [Y - V(J^\pm) /\mathbb{G}_m ].
\]

Now we may use Theorem~\ref{thm:intro1} to  prove the following derived equivalences.

\begin{cor}[Corollary~\ref{cor:ci}]\label{cor:intro} Assume that the dg-algebra generators of $ R $ over $ T $ have internal degree zero. When $ \mu_{+} + \mu_{-} = 0 $, the wall crossing functor
	\[
	\Phi^{\operatorname{wc}} : \operatorname{Perf}(Y\sslash{+}) \longrightarrow \operatorname{Perf}(Y\sslash{-})
	\] is an equivalence of categories.
	In particular, if the semi-stable loci are smooth (for example, the setting of smooth flops),
	\[
	\Phi^{\operatorname{wc}} : D^b(Y\sslash{+}) \longrightarrow D^b(Y\sslash{-})
	\] is an equivalence of categories.
\end{cor}

As a special case we recover the local case of the equivalence in \cite{NamiFlops,Kaw,Hara,Mor}, using VGIT and window techniques.

\begin{thm}[Corollary \ref{cor:mf}] \label{thm:introthm2}
	For the local model of the Mukai flop over a fixed commutative Noetherian ring $ k $, the wall crossing functor
	\[
	\Phi^{wc} : D^b(Y\sslash{+}) \to D^b(Y\sslash{-})
	\] is an equivalence of categories, and the kernel for the equivalence is 
	\[
	\mathcal{O}_{Y\sslash{+} \times_{Y\sslash{\,0}} Y\sslash{-}}
	\]
\end{thm}

 Often, in simple examples of flops such as the Atiyah/standard flops or the Mukai flop, the kernel for the wall-crossing functor is the fiber product of the flop diagram (as shown in Theorem~\ref{thm:introthm2}, for example).    However, this simple Fourier-Mukai kernel is not the right one in general, as one can see in the case of stratified Mukai flops \cite{Cautis}, where the kernel is the pushforward of a locally free sheaf along an open immersion to the fiber product.
 
 Therefore, we want to stress that the kernel we construct is  not always the fiber product. Instead it is related to  the fiber product of a diagram obtained by deriving the scheme   (see Proposition~\ref{prop:fiber} for the precise statement). Indeed, one striking feature of this `derived' construction,  is that the kernel for the  wall-crossing functor described in Corollary~\ref{cor:intro} need not even be a sheaf in general. In Example~\ref{ex:Qnotasheaf}, we show that the kernel is a complex with two homology sheaves that we identify explicitly. Although  the object $ Q $ is very simple as a dg-sheaf on the dg-scheme $ \X \times \X$, it is not easy to describe it as a complex of sheaves on $ Y \times Y $.  We may interpret this observation to mean that the derived construction of $ Q $ hides the complicated nature of the kernel for the wall-crossing functor. (Even further, we think that this observation is a hint that the derived construction of $ Q $ is indeed the right approach!)

\subsubsection{Comparison with the literature}
 
As this paper follows the idea of deriving $ Q $ as presented in \cite{bdf}, a few remarks are in order. As mentioned previously, we use the philosophy of the monoidal Dold-Kan correspondence to work in the dg-setting instead of the simplicial setting of \textit{loc.\ cit.\ }
 
 Firstly, note that the monoidal Dold-Kan correspondence is  a Quillen equivalence between connected cdgas and connected simplicial commutative algebras, over a field of characteristic zero. Hence, if $ k $ is a field of characteristic zero, the functor $ \Phi_{Q_+} $ of Proposition~\ref{prop:ffwf} should agree with the analogous functor defined in \textit{loc.\ cit.\ } by applying the monoidal Dold-Kan correspondence.  However, over an arbitrary Noetherian ring $ k $ (as is the setting of both papers), it is unclear whether or not these functors agree.
 
 Secondly, all of our main results concerning windows and wall-crossings (for example, Theorems~\ref{thm:introQequiv}, \ref{thm:intro1} , \ref{thm:intro2}, \ref{thm:introthm2} and Corollary~\ref{cor:intro}) hinge on an explicit identification of the window subcategory $ \mathbb W $. The paper \cite{bdf} does not identify this window subcategory explicitly (in the context of derived $ Q $), and hence our main results can be seen as the natural continuation of the program presented in \textit{loc.\ cit.\ } 
 
In the  setting of singular VGIT problems (as in Section~\ref{sec:introVGIT}), the cdga $ R $ is chosen to be a semi-free dg-resolution of a singular ring. Under certain assumptions  called Properties $ L(+) $ and $ A $, \cite{HL} also defines a window subcategory and proves that the derived category of the GIT quotient is equivalent to this window (via different methods).

We note that our conditions on the degree in Theorem~\ref{thm:introQequiv} need not satisfy Property $A$ or Property $L^+$ of \cite{HL} (see Example~\ref{ex:degreef}). In fact, these conditions are roughly complementary to the ones of \cite{HL}. Moreover, in Example~\ref{ex:twopoints}, where $L^+$ and $A$ are satisfied (but our Theorem~\ref{thm:introQequiv} does not hold), we check that the essential image of our window functor $ \Phi_{Q_+} $ is the same as the window described in \cite{HL}. In fact, in future work, we will show that this is true in general.   This suggests  that the Fourier-Mukai kernel that we study in this paper is indeed the right one (for example, to solve the Bondal-Orlov-Kawamata conjectures).

\subsection{Outline of the paper}
In Section~\ref{sec:2}, we first set up notation and introduce the language of dg-schemes. We discuss the derived category of equivariant dg-schemes and the existence of derived functors in Section~\ref{sec:dg}. We define the category of perfect complexes in Section~\ref{sec:perf} which is the category we are interested in, and in Section~\ref{sec:VGIT}, we introduce the Variation of GIT quotients problem that we study in this paper.

In Section~\ref{sec:3}, we define the object $ Q $ which is the main object of this paper, and study various properties of it. Section~\ref{sec:Qa} is a brief reminder of the construction of $ Q $ in the setting of (non-dg) rings, and then we define it for semi-free cdgas in Section~\ref{sec:Qcdga}. We study various properties of $ Q $ in Section~\ref{sec:Qprop}, which will be important to understand the Fourier-Mukai transforms induced by it. In Section~\ref{sec:ff}, we find conditions for fully-faithfulness of the window functor defined using $ Q $.

In Section~\ref{sec:4}, we study the window functors and wall-crossing functors defined using $ Q $ as the Fourier-Mukai kernel. We focus on the case where the underlying scheme is affine space in Section~\ref{sec:aff} and the general case of a smooth affine scheme in Section~\ref{sec:gen}. Then, we study the induced wall-crossing functors in Section~\ref{sec:wc} and finally discuss applications to VGIT problems arising from flops in Section~\ref{sec:flops}, and comment on the Fourier-Mukai kernel for the wall-crossing functor.

\subsection{Acknowledgments}  We are very grateful to M. Ballard for many discussions on this work from start to finish (i.e.,\ including its initialization).  We  wish to thank S. Riche for helpful discussions regarding derived categories of  dg-schemes, and T. Bridgeland for comments about the nature of the kernel for wall-crossing functors for flops. We thank J. Rennemo for pointing out a typo regarding perfect objects that appeared in Proposition~\ref{prop:ffwf} in a previous version.  We are grateful to the referees for various helpful comments and corrections. The authors were partially supported by the NSERC Discovery Grant and CRC program.

\section{Setup}\label{sec:2}
  Let us set notation and recall relevant definitions and results about dg-schemes following  \cite{CFK,Riche, MR, BR}. 
  
  \subsection{Notation}
  Throughout, $ k $ will denote a fixed commutative Noetherian ring over which all objects are defined. In this paper, we will only work with commutative differential graded algebras concentrated in non-positive degrees, and henceforth we will simply refer to one as a cdga. Bold letters, e.g.\ $ \mathbf X $, will denote  a dg-scheme, and the underlying ordinary scheme will be denote by unbolded letters, e.g.\ $ X $. We will assume that all our schemes are separated and Noetherian of finite dimension. 
  
  	Often, our objects (cdgas, for example) will have a $ \mathbb{Z} $-grading coming from a $ \mathbb G_m $-action (on dg-schemes, for example) which we will refer to as the \textit{internal} grading, as opposed to the homological grading coming from the cdga structure. In order to denote the homological grading, we will use upper indices, whereas we use lower indices for the internal grading.  
	
	We say that the graded vector space $k(i)$ has weight $i$.  For a graded module $M$, the degree $0$ piece of the shifted module $M(i)$ is the degree $i$ piece of $M$, notated $M_i$ so that  $(M(i))_j = M_{i+j}$.     
   
   \subsection{dg-schemes}\label{sec:dg}
   
    The data of a  dg-scheme $ \mathbf{X} $ is the pair
   \[
   \mathbf{X} = (X , \mathcal{A}),
   \] where $ X $ is a scheme and $ \mathcal{A} $ is a non-positively graded, commutative dg $ \mathcal{O}_X  $-algebra, such that $ \mathcal{A}^i $ is a quasi-coherent $ \mathcal O_X $-module for any $  i \in \mathbb{Z}_{\leq0} $. We denote the homological graded pieces of $ \mathcal A $ by  $ \mathcal A^i $.

  \begin{defin} Let $ \mathbf X  = (X, \mathcal{A}) $ be a dg-scheme. A \textit{quasi-coherent} dg-sheaf $ \mathcal{F} $ on  $ \mathbf{X} $ is a $ \mathcal A $-dg-module such that $ \mathcal{F}^i $ is a quasi-coherent $ \mathcal O_X $-module for any $ i \in \mathbb{Z} $.

   \end{defin}

We define the derived category of quasi-coherent dg-sheaves by $ D(\QCoh \mathbf{X}) $. The derived category is defined in the usual way as the localization of the homotopy category with respect to the class of quasi-isomorphisms. This has the structure of a triangulated category. Note that in the case of ordinary schemes,  i.e., if $ \mathcal A = \mathcal O_X $, the category $ D(\QCoh \mathbf{X}) \cong D (\QCoh X) $.

 In \cite{BR}, it is shown  that there are enough $ K $-flat and $ K $-injective objects in $ \QCoh(\X) $, and hence one can use these to define all right derived functors, and left derived tensor products and pullbacks. They also prove the adjunction between derived pushforwards and  pullbacks,  the projection formula and a base change formula. It is worth noting that in the context of dg-schemes, flatness is not required for the base-change formula; instead we use the derived fiber product.

 In the $ \mathbb{G}_m $-equivariant setting, \ie when we have a scheme $ X $ equipped with a $\mathbb G_m $-action and $ \mathcal A $ is equivariant with respect to this $\mathbb G_m $-action, some of the appropriate generalizations are defined by \cite{MR} under the following technical assumptions:
 
 \begin{enumerate}
 	\item For any $ \mathcal F $ in $ \QCoh^{\mathbb G_m}(X) $, there exists a  $ P $ in $ \QCoh^{\mathbb G_m}(X) $ which is flat over $ \mathcal O_X $ and a surjection $ P \twoheadrightarrow \mathcal  F $ in $ \QCoh^{\mathbb G_m}(X) $.
 	\item  Assume that $ \mathcal A $  is locally free  over $ \mathcal A^0 $,  $ \mathcal A^0 $ is locally finitely generated as an $ \mathcal O_X $-algebra, and finally that $ \mathcal A $ is K-flat as a $\mathbb G_m $-equivariant $ A^0$-dg-module.
 \end{enumerate}
 
 In this paper, we will always work over $ X $ quasi-affine or affine, and hence the first condition is automatically satisfied. We will only consider dg-schemes such that $ \mathcal A^0  = \mathcal O_X $ and that $ \mathcal A $ is semi-free over $ \mathcal O_X $, and hence the second condition is always satisfied as well.
 
 We note that \cite{MR} imposes a slightly stronger technical assumption which is that $ \mathcal A $  is locally free  of finite rank over $ \mathcal A^0 $. However, this assumption is unnecessary in the proof of \cite[Proposition 2.8]{MR}. Moreover, the proof of \cite[Theorem 1.3.6]{Riche} carries over to this  setting as well. Hence we have the following statement.
 
 \begin{lem}
 	Consider a $ \mathbb G_m $-equivariant dg-scheme $ \mathbf X $ satisfying the assumptions above. For any object $ \mathcal F $ in $ \QCoh^{\mathbb G_m} (\X) $, there exists a K-injective equivariant resolution $ \mathcal F \to \mathcal  I $.
 \end{lem}
\begin{proof}
	First we use \cite[Proposition 2.8]{MR} to find a resolution for bounded below $ \mathcal A $-dg-modules; this resolution has flabby graded components. Then, the proof in \cite[Theorem 1.3.6]{Riche} carries over to this setting. As a brief reminder, we truncate the dg-sheaf $ \mathcal F $, resolve the truncations by K-injectives and then take the inverse limit over the truncations. In the proof of \cite[Theorem 1.3.6]{Riche}, we take $ \mathfrak B = \QCoh^{\mathbb G_m} (X)$ and consider equivariant open covers of $ X $.
\end{proof}

Using the above lemma, we can define all right derived functors. In order to define left derived inverse images and tensor products, we use \cite[Lemma 2.7]{MR} which proves the existence of K-flat resolutions. Using these K-injective and K-flat resolutions it is easy to prove the projection formula and the adjunction between pushforwards and pullbacks  by adapting the arguments of \cite{BR} to our setting, and we will use these in this paper.

 \subsection{Perfect objects}\label{sec:perf}
 
 Let us give an ad-hoc definition of the category $ \Perf^{\mathbb G_m} \X $ where $ \X = (X, \mathcal A) $ with the assumptions as before, and $ X $ is quasi-projective. Consider an ample line bundle $ \mathcal L $ on $ X $.
 
 \begin{defin}\label{def:perf}
 	The thick triangulated subcategory of $ D(\QCoh^{\mathbb G_m} \X) $ generated by \hbox{$ \mathcal  L^{\otimes j} \otimes_{\mathcal O_X} \mathcal A (i) $} for all $ i,j \in \mathbb Z $ is called the \textit{category of perfect dg-sheaves}, denoted $ \Perf^{\mathbb G_m} \X $.
 \end{defin}

 \begin{rem}\label{rem:perf}
 	We expect that the `right' definition of the category of perfect dg-sheaves on $ \X = (X,\mathcal A) $, i.e.\ the triangulated category generated by dg-sheaves that are  quasi-isomorphic to a finite locally semi-free $ \mathcal{A} $-module (\ie finite locally free after forgetting the differential) is equivalent to our ad-hoc definition in this context. In particular, it is  well known that if we consider an ordinary scheme (if $ \mathcal{A} = \mathcal O $) equipped with an ample line bundle then the definitions coincide.
 \end{rem}

 \subsection{GIT problem}\label{sec:VGIT}
 
 We start with a $ \mathbb{Z} $-graded semi-free\footnote{Recall that a semi-free cdga $ R $ is one that is free when considered as a commutative algebra by forgetting the differential.} commutative differential graded algebra (cdga), say $ R $. We denote the homological degree zero piece of $ R $ by $ T := R^0 $, and we require that $ T $ is a finitely generated smooth ring over $ k $. We do not require any finiteness conditions on $ R $. To this data, we associate the affine dg-scheme 
 \[
  X := \Spec T, \qquad \X := (X , \mathcal R),
 \] where $ \mathcal R $ is the dg-$ \mathcal O_X $-module associated to the cdga $ R $. The $ \mathbb Z $-grading on $ T \to R $ is equivalent to a $ \mathbb G_m$-action on $ \X $.
 
 We denote the semi-stable loci (which are dg-schemes) as
 \[
	 \X^\pm := (X^\pm, \mathcal R^\pm),
 \] where we define 
 \[
  X^{\pm} := X - V(J^\pm), \qquad \mathcal R^\pm =  \mathcal R|_{X^\pm}.
 \] Here  $ J^\pm $ denotes the ideal generated by all the strictly positive/negative-ly  graded elements. Then we define  two Geometric Invariant Theory (GIT) quotients $ \X \sslash{\pm} $ of $ \X $ with respect to the $ \mathbb{G}_m $-action. The stacks  
 \begin{align*}
  \qquad  X \sslash{\pm} := [X^{\pm} / \mathbb{G}_m],
 \end{align*}with the sheaf of cdgas obtained by  the descent of $ \mathcal R^\pm $ to $ X\sslash{\pm} $ are the  GIT quotient ``dg-stacks" 
\[
\X \sslash{\pm} = (X\sslash{\pm}, \mathcal R^\pm).
\]

\begin{rem}
	We do not define or introduce dg-stacks in general as the ones we consider are global quotients of dg-schemes, and  we are only concerned about the corresponding derived categories. We define the derived category of the global quotient dg-stack merely as the corresponding equivariant derived category of the dg-scheme. 
\end{rem}

\section{Defining Q}\label{sec:3}
We will first recall the $ Q $-construction briefly in the setting of $ \mathbb{Z} $-graded  (not dg-) rings, and then  extend this definition to cdgas. In this section, we continue to use the notations of the previous section. In particular, we remind the reader that  $ T $ is a $ \mathbb{Z} $-graded smooth ring finitely generated over $ k $,  $ R $ is a $ \mathbb Z $-graded  semi-free cdga over $ T $ and $ \X $ is an  affine dg-scheme (associated to $ R $) equipped with a $ \mathbb G_m $-action.
 
 \subsection{$ Q $ for smooth rings}\label{sec:Qa}

This section is a brief reminder of the construction of \cite{bdf}. The $ \mathbb Z $-grading on the ring $ T $ is equivalent to a $ \mathbb G_m $-action on $  X  = \Spec T $. Viewed in this manner,  $ T $ is equipped with the co-projection and co-action morphisms, which we denote by $ \pi $ and $ \sigma $ respectively. These morphisms act as 
\begin{equation}\label{eq:pisigmarings}
\begin{aligned}[c]
\pi :& T \to T[u,u^{-1}]\\
& t \mapsto t
\end{aligned}
\qquad \qquad
\begin{aligned}[c]
\sigma:& T \to T[u,u^{-1}] \\
&t \mapsto t u^{\operatorname{deg} t}  ,
\end{aligned}
\end{equation} for any homogeneous element $ t $ in $ T $. We assign a $ \mathbb Z \times 
\mathbb Z $-grading to the equivariant diagonal of $ \Spec T $, which is
\[
k[\mathbb G_m \times X] = T[u,u^{-1}],
\] such that  the morphisms $ \pi $ and $ \sigma $ equivariant. The  grading for homogeneous $ t $ in $ T $ is defined as $$ \operatorname{deg} \pi(t) = (\operatorname{deg} t, 0 ),\,\,\, \operatorname{deg} \sigma(t) = (0,\operatorname{deg} t ) \,\text{ and } \operatorname{deg} u = (-1,1). $$

 Given the smooth ring $ T $, we define
\[
Q(T) := \langle \pi(T), \sigma(T), u \rangle \subseteq T[u,u^{-1}],
\] 
as the $ k $-subalgebra of $ T[u,u^{-1}] $ generated by the image of the co-action and co-projection maps, and $ u $. It suffices to keep the image of the negative elements under the co-action map along with $ T[u] $,
\[
Q(T) = \langle \bigoplus_{i< 0} T^iu^i, T[u]\rangle
\] where $ T^i $ denotes the $ i $-th graded piece of $ T $.
The  co-projection and co-action  maps factor through $ Q(T) $, giving the maps $ p $ and $ s $,
\begin{equation}\begin{tikzcd} \label{eq:ptst}
T \arrow[r,shift left,"p"] \arrow[r,shift right,swap, "s"]& Q(T)\arrow[r,hook] & \Delta(T).
\end{tikzcd}
\end{equation} 

The $ T \otimes T $-module $ Q(T) $ also inherits the $ \mathbb{Z}\times \mathbb{Z} $ grading from $ T[u,u^{-1}]  $. Hence, we can consider $ Q(T) $ as an element of $ D(\operatorname{mod}^{\mathbb{G}_m \times \mathbb{G}_m}(T \otimes T)) $ using the $ p \otimes s $ module structure.

In order to clarify the construction above, let us consider the  example where $ X $ is an affine space. This example will  play an important role in the paper.

\begin{exa}\label{ex:affineT}
	Let us consider the  ring $ T = k[\mathbf{x}, \mathbf{y}] $, where we are using the shorthand notation $ \mathbf{x} $ to mean $ x_1,\cdots,x_l $, and $ \mathbf y $ to mean $ y_1,\cdots, y_m $. We assign the $ \mathbb Z $-grading
	\[
		\operatorname{deg} x_i = a_i > 0 \qquad \operatorname{deg} y_i = b_i < 0
	\] to $ T $. Then, $$ Q(T) = k[\mathbf{x},\mathbf{z},u], $$ with the $ p,s : T \to Q(T) $ maps given by
	\begin{equation*}
	\begin{aligned}[c]
	p(x_i) &= x_i \\
	p(y_i) &= u^{-b_i} z_i
	\end{aligned} \qquad
	\begin{aligned}[c]
	s(x_i) &= u^{a_i}x_i \\
	s(y_i) &=  z_i
	\end{aligned}
	\end{equation*}	
\end{exa}

\begin{rem}
 For the sake of notational simplicity, we have not added any $ x_i $ or $ y_j $ of internal degree zero. The $ Q $ construction does not affect such degree zero generators and hence we may view it as a part of the Noetherian ring $ k $.
\end{rem}

There is a geometric motivation for this definition of $ Q $, and we refer the reader to \cite{bdf} for more details. The idea is to define $ Q $ as a partial compactification of the $ \mathbb{G}_m $-action on $ X $, generalizing a construction of \cite{drin}. In \textit{loc.\ cit.}, the authors prove various properties of this object $ Q $; one that is worth mentioning here is that the assignment \hbox{$ Q : \kcr \to \operatorname{CR}^{\mathbb{G}_m\times\mathbb{G}_m}_{k[u]} $} is functorial. Here, $ \kcr $ denotes the category of $ \mathbb Z $-graded commutative algebras over $ k $, and $ \operatorname{CR}^{\mathbb{G}_m\times\mathbb{G}_m}_{k[u]} $  denotes the  category of $ \mathbb Z \times \mathbb Z $-graded commutative algebras over $ k[u] $. We also note that the equivariant injective map $ Q(T) \hookrightarrow \Delta(T) $ provides a natural transformation between the Fourier-Mukai functors $ \Phi_{Q(T)} \to  \Phi_{\Delta(T)} =\operatorname{Id}_X $.
\subsection{$ Q $ for semi-free cdgas}\label{sec:Qcdga}

Now we define $ Q $ for semi-free cdgas, which is the case of interest in this paper. Recall that the dg-scheme $ \X = (\Spec T, \mathcal R) $ is equipped with a $ \mathbb G_m $-action. This means that we have the projection and action maps, which we denote by $ \pi $ and $ \sigma $ respectively,
\[
\begin{tikzcd}
 {\mathbb{G}_m} \times  \mathbf X  \arrow[r, shift right, swap, "\pi"] \arrow[r, shift left, "\sigma"] & \mathbf X.
\end{tikzcd}
\] 
We want to view the cdga associated to $ {\mathbb{G}_m} \times  \mathbf X $ as a dg-sheaf on $ \X \times \X $ using the morphism $ \pi \times \sigma $. More precisely, we define a semi-free cdga $ \Delta(R) $ (which should be thought of as the equivariant diagonal in the dg-setting)
\[
\Delta(R) :=  R[u,u^{-1}],
\] which is equipped with a $ \mathbb{Z} \times \mathbb{Z} $-grading (which, as before, is chosen in order to make the action and projection maps equivariant). The  grading for homogeneous $ r $ in $ R $ is defined as 
$$ \operatorname{deg} \pi(r) = (\operatorname{deg} r, 0 ), \,\,\, \operatorname{deg} \sigma(r) = (0,\operatorname{deg} r ) \, \text{ and } \operatorname{deg} u = (-1,1). $$
 In order to get the dg-structure, we note that the elements, $ u $ and $ u^{-1} $ are in homological degree zero, and hence are killed by the differential.

We also have the co-action and co-projection dg-morphisms (which we denote by $ \pi $ and $ \sigma $ by abuse of notation)
\begin{equation}\label{eq:pisigma}
\begin{aligned}[c]
\pi :& R \to R[u,u^{-1}]\\
& r \mapsto r
\end{aligned}
\qquad \qquad
\begin{aligned}[c]
\sigma:& R \to R[u,u^{-1}] \\
&r \mapsto r u^{\operatorname{deg} r}  ,
\end{aligned}
\end{equation} for homogeneous $ r $ in $ R $. 

Consider  $ \Delta(R)  $ as a $ R \otimes R $-module  with the module structure  $  \pi \otimes \sigma $. Then,  we have the associated quasi-coherent dg-sheaf, also denoted $ \Delta(R) $, on $ \X \times \X $. Taking into account the $ \mathbb{Z} \times \mathbb{Z}$-grading we view it as  an element of $ D(\QCoh^{\mathbb G_m \times \mathbb G_m} \mathbf X \times \mathbf X) $. 

\begin{lem}
	The object $ \Delta(R) \in  D(\QCoh^{\mathbb G_m \times \mathbb G_m} \mathbf X \times \mathbf X)  $  is a Fourier-Mukai kernel of the identity functor on $ D(\QCoh^{ \mathbb G_m} \mathbf X ) $.
\end{lem} 
\begin{proof} As all the objects are affine, we may work on the level of cdgas. Then, keeping track of the equivariant structures carefully, we have the following sequence of isomorphisms:
	\begin{align*}
	\Phi_{\Delta(R)} (A) &\cong (\sigma_* \pi^*A)^{\mathbb G_m} \\
	&\cong (A[u,u^{-1}])_{(0,*)} \\
	&\cong \oplus A^i u^i,
	\end{align*} where the $ R $-module structure is
	\[
	r.(au^i) = (r.a) u^{i+\operatorname{deg} r}.
	\]

	Similarly, for the inverse functor,
	\begin{align*}
	\Phi_{\Delta(R)} (A) &\cong (\pi_* \sigma^*A)^{\mathbb G_m} \\
	&\cong (A[u,u^{-1}])_{(*,0)} \\
	&\cong \oplus A^iu^{-i},
	\end{align*}where the $ R $-module structure is
	\[
	r.(au^i) = (r.a) u^{i-\operatorname{deg} r}.
	\]
	
	It is clear that $ \oplus A^i u^i $ and $ \oplus A^iu^{-i}  $ are  isomorphic to $ A $ with the original $ R $-module structure, and hence we are done.
\end{proof}

Now, we define an analogue of the $ Q $-construction of \cite{bdf} for the semi-free cdga $ R $.

\begin{defin}
	The object $ Q(R) $ is defined as the following $ \mathbb{Z} \times \mathbb{Z} $-graded cdga
	\[
	Q(R) = \langle \pi(R), \sigma(R), u \rangle \subseteq R[u,u^{-1}]
	\] generated as a $ k $-subalgebra  of $ R[u,u^{-1}] $. The  dg-structure and the $ \mathbb Z \times \mathbb Z $-grading are inherited from the object $ \Delta(R) = R[u,u^{-1}] $.

\end{defin} 

 The  co-action and co-projection maps to $ \Delta(R) $ factor through $ Q(R) $, giving the maps $ p $ and $ s $,
\[\begin{tikzcd}
R \arrow[r,shift left,"p"] \arrow[r,shift right,swap, "s"]& Q(R)\arrow[r,hook] & \Delta(R).
\end{tikzcd}
\] The maps $ p,s $ extend to $ \Delta(R) $ to give the co-projection $ \pi $ and co-action $ \sigma $ maps on $ R $ respectively.

Let us clarify these notions by looking at an extension  of Example~\ref{ex:affineT}.
\begin{exa}\label{ex:3.5}
	Let $ T = k[\mathbf x, \mathbf y] $ as in Example~\ref{ex:affineT}. Consider a homogeneous regular sequence of length $ n $, say $ (h_1,h_2,\cdots, h_{n_1+n_2}) $ in $ T $. Assume further, without loss of generality, that the internal degree of $ h_i $ is $ d_i $, and that 
	\[
		d_i \geq 0 \text{ for } i \in [0,n_1], \qquad d_i < 0 \text{ for } i \in [n_1+1,n_1 + n_2].
	\] Then, we define the dg-algebra $ R $ to be the Koszul resolution on this regular sequence.
	\[
	R = k[\mathbf x, \mathbf y, e_1,\cdots,e_{n_1}, f_1,\cdots,f_{n_2}], \qquad d e_i = h_i, \, \, d f_j = h_{n_1+j}.
	\] Then we can compute $ Q(R) $ to be
	\[
	Q(R) = k[\mathbf x,\mathbf z, e_1,\cdots, e_{n_1}, g_1,\cdots,g_{n_2}],
	\] where  the maps $ p,s : R \to Q(R)  $ are given by
	\begin{equation*}
	\begin{aligned}[c]
	p(x_i) &= x_i \\
	p(y_i) &= u^{-b_i} z_i \\
	p(e_i) &= e_i \\
	p(f_i) &= u^{-d_{n_1+i}} g_i
	\end{aligned} \qquad
	\begin{aligned}[c]
	s(x_i) &= u^{a_i}x_i \\
	s(y_i) &=  z_i \\
	s(e_i) &= u^{d_i}e_i \\
	s(f_i) &=  g_i
	\end{aligned}
	\end{equation*}	
\end{exa}

 Let us return to the general setting now. The maps $ p,s $ give $ Q(R) $ the structure of  a $ \mathbb{Z} \times \mathbb Z $-graded $ R \otimes R $ dg-module with the $ p \otimes s $-module  structure. We will view the associated dg-sheaf, also denoted  $ Q(R) $, as an element of $ D(\QCoh^{\mathbb G_m \times \mathbb G_m} \mathbf X \times \mathbf X) $.
 
 \subsubsection{Explicit description of $ Q(R) $}
	
	Let us introduce some more notation and describe $ Q(R) $ explicitly. As $ R $ is a semi-free cdga over $ T $, let us choose a set of (possibly infinite) homogeneous (in the internal grading) algebra generators. We denote the positive generators by $ e_i $ and the negative generators by $ f_i $, where $ i $ takes values in a possibly countably infinite set. We will use the short hand notation $ \mathbf e $ ($ \mathbf f $) to refer to the set of all $ e_i $ ($ f_i $). Using this notation,
	\[
	R = T[\mathbf e, \mathbf f].
	\] 
	
	Defining $ g_i : = u^{\operatorname{deg} f_i} f_i$, we can express $ Q(R) $ explicitly as 
		\begin{equation}\label{eq:Q}
		Q(R) = Q(T)[\mathbf e, \mathbf g],
		\end{equation}
	with the $ p $ and $ s $ dg-module structures given by
\begin{equation}
\begin{aligned}[c]
p :& R \to Q(R)\\
& t \mapsto p_T(t) \\
& e_i \mapsto e_i \\
& f_i \mapsto u^{-\dg f_i}g_i
\end{aligned}
\qquad \qquad
\begin{aligned}[c]
s:& R \to Q(R) \\
& t\mapsto s_T(t) \\
&e_i \mapsto u^{\dg e_i} e_i \\
& f_i \mapsto g_i,
\end{aligned}
\end{equation} where the maps $ p_T $ and $ s_T $ are the maps defined in equation~\ref{eq:ptst}, but we have added the subscript $ T $ for clarity.

We also note the bi-degrees of the  elements in $ Q(R) $ for convenience:
\[
  \operatorname{bi-deg} e_i = (\dg e_i, 0 ) \qquad \operatorname{bi-deg} g_i = (0, \dg f_i),
\] and the degrees of the elements in $ Q(T) $ are as defined in Section~\ref{sec:Qa}.

	\subsection{Some properties of $ Q(R) $} \label{sec:Qprop}
Let us denote the $ \mathbb{G}_m \times \mathbb G_m $-equivariant inclusion of $ Q(R) $ into $ \Delta(R) $ by $ \eta $,
\[ 
\eta: Q(R) \hookrightarrow \Delta(R) = R[u,u^{-1}].
\] 
Let us study the relation of $ Q(R) $ to $ \Delta(R) $ further. In particular, we show that they become isomorphic if we localize by an element in $ T $ of non-zero internal degree.

\begin{lem}\label{lem:Qdelta}
	Consider $  Q(R) $ as a $ R \otimes R $ dg-module with the $ p \otimes s  $ structure. Let $ t $ in $ T $ be a homogeneous element, with the corresponding localization map $ R_t = T_t[\mathbf e, \mathbf f] \to R = T[\mathbf e, \mathbf f] $. If $ \deg{t} > 0  $, 
	\[
	1 \otimes_s \eta :   R_t \otimes_s Q(R) \to  R_t \otimes_s \Delta(R)
	\] is an isomorphism. If $ \deg{t} < 0  $, 
	\[
	1 \otimes_p \eta : R_t  \otimes_pQ(R) \to  R_t \otimes_p \Delta(R)
	\] is an isomorphism.
\end{lem}

\begin{proof} The morphisms are injective as $ R_t $ is flat over $ R $.
	
	To check that the morphisms are surjective, we just need to check that we get $ u^{-1} $ in the image. In the first case, $  t^{-1}  \otimes u^{\deg t -1} t $ maps to $ u^{-1 }$; in the second case, $   t^{-1} \otimes s(t) u^{-\deg t-1}  $ maps to $ u^{-1} $.

\end{proof}

We need to study some properties of the object $ Q(R) \sop Q(R) $ here as it will play a role when we discuss fully faithfulness of the window functor in the next subsection. The object $ Q(R) \sop Q(R)  $ inherits a $ \mathbb{G}_m^{\times 3} $-action, where the $ \mathbb{Z}^{ 3} $ grading is as follows:
\[
\operatorname{deg} q\otimes1 = (a,b,0) \qquad  \operatorname{deg} 1\otimes q = (0,a,b) 
\] if $ q \in Q(R) $ is a homogeneous element of degree $ (a,b) $. 

We would like to understand what happens when we take middle degree invariants; we denote this by $ ( M )_{0} $ where $ M $ is $ \mathbb{Z}^{ 3} $-graded.

\begin{lem}\label{lem:QQ0}
	The following diagram commutes
	\[\begin{tikzcd}
	& (Q(R) \; {}_s\!\otimes_\pi \Delta)_0 \arrow[rd, "\sim"] & \\
	(Q(R) \sop Q(R))_0 \arrow[ru,] \arrow[rd,]& & Q(R) \\
	& (\Delta \; {}_\sigma \otimes_p Q(R))_0 \arrow[ru,swap,"\sim"]
	\end{tikzcd}	\]
\end{lem}
\begin{proof}
	Consider the upper right arrow
	\begin{align*}
	(Q(R) \; {}_s\!\otimes_\pi \Delta)_0 &= Q(R)[v,v^{-1}]_0 \\
	&\cong (\langle \bigoplus_{i\leq 0} R_i u^i, R[u] \rangle[v,v^{-1}])_0 \\
	&\cong \langle \bigoplus_{i\leq 0} R_i u^iv^i , R[uv] \rangle \\
	&\cong Q(R).
	\end{align*}
	
In order to take the middle degree zero invariants, we used that $ \operatorname{deg} u = (-1,1,0) $ and $ \operatorname{deg} v = (0,-1,1) $. Similarly, one can prove the same for the bottom row, and the commutativity is clear.
\end{proof}

We consider the following  morphism,
\begin{equation}\label{eq:rho}
\rho : (Q(R) \sopl Q(R))_0 \to (Q(R) \sop Q(R))_0 \to Q(R)
\end{equation}
 where the first morphism is the map from the left derived functor to the (underived) functor, and the second morphism is the one constructed in Lemma~\ref{lem:Qdelta}.
Fully faithfulness of the window functor is related to properties of the morphism $ \rho $ as we will see in Section~\ref{sec:ff}.

Our goal is to construct a Fourier-Mukai kernel for the window functors 
\[
D(\operatorname{QCoh}^{\mathbb{G}_m}\X^+) \longrightarrow D\left(\operatorname{QCoh}^{\mathbb{G}_m} \X \right).
\] Using the $ \mathbb{G}_m $-equivariant morphism of  dg-schemes
\[
j : \X^+ \to \X,
\]
 we define
\[
Q_+ :=  (j\times\operatorname{Id})^*Q(R),
\] and consider it as an object of $ D(\operatorname{Qcoh}^{\mathbb{G}_m \times \mathbb{G}_m}\X^+ \times \X)  $. We have dropped the $ R $ in $ Q_+ $ in the interest of notational convenience. We will focus on the positive GIT quotient, but the arguments are analogous for the negative GIT quotient.

\subsection{Fully faithfulness}\label{sec:ff}
In this section, we find sufficient conditions for  the Fourier-Mukai functor $ \Phi_{Q_{+}} $ to be fully faithful. Checking faithfulness is easy and follows the arguments of \cite{bdf}.

\begin{lem} \label{lem:faithfulness}
	The composition $ j^* \circ \Phi_{Q_+}$ is naturally isomorphic to the identity.  In particular, the functor 
	\[
	\Phi_{Q_+} :  D(\operatorname{Qcoh}^{\mathbb{G}_m} \X^+) \longrightarrow D\left(\operatorname{QCoh}^{\mathbb{G}_m} \X \right),
	\]
	is faithful.
\end{lem}
\begin{proof}
	We want to show that $ j^* \circ \Phi_{Q_+} $ is the identity functor. On the level of the kernels, it suffices to show that the morphism
	\[	
		(j \times j)^*Q \to (j \times j)^* \Delta(R)
	\] is an isomorphism. We can do this locally on the $ \mathbb G_m $-invariant affine cover of $ X^+ \times X^+ $, obtained by inverting the positive elements $ t \in J^+ $ in $ T = k[X] $. This is precisely the content of the first part of Lemma~\ref{lem:Qdelta}.
\end{proof}

Fullness of the functor is more involved and is best phrased in the language of Bousfield localizations. Let us recall some of the definitions and results that we need about Bousfield (co)-localizations.  The existence of a Bousfield triangle produces a semi-orthogonal decomposition, and we show that the essential image of our functor is equivalent to one part of the semi orthogonal decomposition.

\begin{defin}
	Let $\mathcal{T}$ be a triangulated category. A \emph{Bousfield localization} is an exact endofunctor $L:\mathcal{T}\ra\mathcal{T}$ equipped with a natural transformation $\delta:1_{\mathcal{T}}\ra L$ such that:
	\begin{enumerate}
		\item[a)] $L\delta=\delta L$ and
		\item[b)] $L\delta:L\ra L^2$ is invertible.
	\end{enumerate}
	A \emph{Bousfield co-localization} is given by an endofunctor $C:\mathcal{T}\ra\mathcal{T}$ equipped with a natural transformation $\epsilon:C\ra 1_{\mathcal{T}}$ such that:
	\begin{enumerate}
		\item[a)] $C\epsilon=\epsilon C$ and 
		\item[b)] $C\epsilon:C^2\ra C$ is invertible.
	\end{enumerate}
\end{defin}

\begin{defin}\label{triangle}
	Assume there are natural transformations of endofunctors
	\begin{equation*}
	C\overset{\epsilon}\ra1_{\mathcal{T}}\overset{\delta}\ra L
	\end{equation*}
	of a triangulated category $\mathcal{T}$ such that
	\begin{equation*}
	Cx\overset{\epsilon_{Cx}}\longrightarrow x\overset{\delta_x}\longrightarrow Lx
	\end{equation*}
	is an exact triangle for any object $x$ of $\mathcal{T}$. Then we refer to $C\ra1_{\mathcal{T}}\ra L$ as a \emph{Bousfield triangle} for $\mathcal{T}$ when any of the following equivalent conditions are satisfied:
	\begin{enumerate}
		\item[1)] $L$ is a Bousfield localization and $C(\epsilon_x)=\epsilon_{C_x}$
		\item[2)] $C$ is a Bousfield co-localization and $L(\delta_x)=\delta_{L_x}$
		\item[3)] $L$ is a Bousfield localization and $C$ is a Bousfield co-localization.
	\end{enumerate}
\end{defin}

\noindent
For a proof that the above properties are indeed equivalent, we refer the reader to \cite[Definition 3.33]{bdf}.

\begin{rem}\label{rem:FMa}
	Any Fourier-Mukai functor $ \Phi_P $ with a morphism $ \Delta \to P $ satisfies the condition a) to be a Bousfield localization. It is easy to see that $ \Phi_P(\delta_A) = \delta_{\Phi_P(A)} $. Analogously, any Fourier-Mukai functor $ \Phi_{P'} $ with a morphism $  P' \to \Delta $ satisfies the condition a) to be a Bousfield co-localization. 
\end{rem}

\begin{lem}\label{lem:BT1}[Property P]
	The triangle of functors 
	\[\begin{tikzcd}
	\Phi_{Q(R)} \arrow[r,] & \operatorname{Id} \arrow[r,]& \Phi_{\operatorname{cone}(\eta)}
	\end{tikzcd}
	\] is a Bousfield triangle if the morphism 
	\[
	\rho : (Q(R) \sopl Q(R))_0 \to Q(R)  ,
	\] is an isomorphism.
\end{lem}
\begin{proof}
	The proof follows the arguments of the proof of \cite[Lemma 3.3.6]{bdf}.
	
	We have a morphism $ Q(R) \to \Delta(R) $, and a morphism $ \Delta \to \operatorname{cone}(\eta)(R) $. Hence (using Remark~\ref{rem:FMa}), we only need to check the second condition for $ \Phi_{Q(R)} $ to be Bousfield co-localization. This condition translates to 
	\[
	(Q(R) \sopl Q(R))_0 \cong Q(R).
	\] 
\end{proof}
If $ Q(R) $ satisfies the  condition of Lemma~\ref{lem:BT1}, we will say that $ Q(R) $ satisfies \hbox{\textit{Property $ P $}}. 

We  have another Bousfield triangle given as follows.

\begin{lem}\label{lem:BT2}
	The following triangle is a Bousfield triangle
	\[ \begin{tikzcd}
	 \Phi_{\operatorname{cone(\gamma)[-1]}}  \arrow[r,] & \operatorname{Id} \arrow[r, ] & j_* \circ j^* ,
	\end{tikzcd}
	\]  where $ \gamma $ is the  morphism $ \Delta \to (\operatorname{Id} \times j)_*(\operatorname{Id} \times j)^*\Delta $.
\end{lem}

\begin{proof}
	First, we check that $ j_*j^* $ is a Bousfield localization. We have the pair of adjoint functors $ j^* \dashv j_* $ . We apply  base change to the following Cartesian square
	\[\begin{tikzcd}
	\X^+ \arrow[r,"\operatorname{Id}"] \arrow[d,"\operatorname{Id}"] & \X^+ \arrow[d,"j"]\\
	\X^+ \arrow[r,"j"] & \X,
	\end{tikzcd}
	\] to see that the co-unit map
	\[
	j^*j_* \overset{\sim}{\to} \operatorname{Id},
	\] is an isomorphism.
	Now, we can show that the functor $ j_*j^* $ equipped with the unit  $ \operatorname{Id} \to  j_*j^* $ is a Bousfield localization (using base change for example). This is similar to the arguments of \cite[Example 1.2]{HR17}.

Finally, we take the cone of the unit natural transformation to get the desired Bousfield triangle
	\[
	\Phi_{\operatorname{cone(\gamma)[-1]}}  \to \operatorname{Id} \to j_* \circ j^*.
	\] Notice that we may do this on the level of the Fourier-Mukai kernels as we have explicit kernels.
\end{proof}

For convenience, we define $ C := \Phi_{\operatorname{cone(\gamma)[-1]}} $.

\begin{lem}\label{lem:FF} If $ Q(R) $ satisfies property $ P $, there is a weak semi-orthogonal decomposition 
	\begin{equation*}
	D(\operatorname{QCoh}^{\mathbb{G}_m}\X)=\left\langle  \operatorname{Im}\Phi_{\operatorname{cone}(\eta)},  \ \operatorname{Im}\Phi_{Q_{+}},  \operatorname{Im}\left(\Phi_{Q(R)}\circ C\right)\right\rangle,
	\end{equation*}
	where $  \operatorname{Im}$ denotes the essential image. Furthermore, the functor 
	\[
	\Phi_{Q_+} : D(\operatorname{Qcoh}^{\mathbb{G}_m}\X^+) \to D\left(\operatorname{Qcoh}^{\mathbb{G}_m}\X\right) \]
	is fully faithful.
\end{lem}

\begin{proof}

	We want to use \cite[Lemma 3.3.5]{bdf} to prove this statement. This says that if $ C_1 \to \operatorname{Id} \to L_1 $ and $ C_2 \to \operatorname{Id} \to L_2 $ are Bousfield triangles in a triangulated category $ \mathcal{T} $ such that $ L_1 C_2 \to L_1 $ is an isomorphism, there is a weak semi-orthogonal decomposition 
	\[
	\mathcal{T} = \langle \operatorname{Im} C_2 \circ L_1, \operatorname{Im} C_2 \circ C_1 , \operatorname{Im} L_2 \rangle.
	\]
	
	This induces a fully-faithful functor
	\[
	F: \mathcal{T}/ \operatorname{Im} C_1 \to \mathcal{T}.
	\]

	We shall take the two Bousfield triangles from Lemma~\ref{lem:BT2} and Lemmas~\ref{lem:BT1} to be the triangles $ 1 $ and $ 2 $ above respectively. Hence we need to show that
	\begin{align*}
	j_* \circ j^* \circ \Phi_{Q(R)} = j_* \circ j^*.
	\end{align*}
	
	The Fourier-Mukai transform $ j^* \circ \Phi_{Q(R)}  $ is the one induced by the kernel $ (\operatorname{Id} \times j)^*Q(R)$. Note that the $ R \otimes R $ module structure on $ Q(R) $ is  $ p \otimes s $.  By Lemma ~\ref{lem:Qdelta} applied to an open cover of $ X \times X^+ $, $ (\operatorname{Id} \times j)^*Q(R) $ is isomorphic to $ (\operatorname{Id} \times j)^*\Delta(R)  $, as we are inverting a positive element with the $ s $-module structure.  Noting that $ \Phi_{(\operatorname{Id} \times j)^*\Delta(R) }  = j^*$ proves the claim.

	The functor $ F $ mentioned above is the functor $ \Phi_{Q_+} $.
\end{proof}

\begin{rem}
	We  mention here that \cite{bdf} proves a semi-orthogonal decomposition result for the homotopy category of the category of spectra in simplicial graded modules over $ R $ (viewed as a simplicial graded commutative ring), which is morally equivalent to a part of the semi-orthogonal decomposition (Lemma \ref{lem:FF}) that we prove. The relevant result is \cite[Proposition 5.4.7]{bdf} for the interested reader.
\end{rem}

\begin{rem}
Notice that the 3-term semi-orthogonal description of Lemma~\ref{lem:FF} closely resembles the one in \cite[Theorem 2.10]{HL}, with which the author defines the window subcategory. In future work, we will show that  our semi-orthogonal description generalizes the one given in \textit{loc.\ cit.\ } By resolving singular schemes as semi-free dg-schemes, we are able to lift the restrictions imposed by properties $ (L+) $ and $ (A) $ that appear in the result of  \cite{HL} (in our context of $ \mathbb G_m $-actions).
\end{rem}

\section{Windows and Wall-crossings} \label{sec:4}
 In this section, we study the object $ Q_+ $ defined earlier as a Fourier-Mukai kernel for the window functor,
\[
\Phi_{Q_+}  : \operatorname{Perf}^{\mathbb G_m}(\X^+) \longrightarrow \operatorname{Perf}^{\mathbb G_m}(\X).
\] We use the various properties of $ Q $ proved in the previous section to prove that the functor is always fully faithful. Then, we identify the essential image of this functor explicitly under certain assumptions on our dg-scheme $ \X $, which proves our first main theorem, Theorem~\ref{thm:intro1}. Using these results, we analyze wall-crossing functors between the GIT quotients and prove that it is an equivalence in some cases, and this gives our second main result Theorem~\ref{thm:intro2}. We also study the cases where the wall-crossing functor is not an equivalence, but only fully-faithful. Then, under certain conditions on $ \X $, we find a semi-orthogonal decomposition for a GIT quotient, where one of the pieces of the SOD is the derived category of the other GIT quotient. 

 For ease of exposition, we first tackle the case that $ X $ is affine space (\ie Example~\ref{ex:affineT}) in  Section~\ref{sec:aff}, as the calculations are more explicit. The case of general smooth affine schemes $ X $ will reduce to this using the Luna slice theorem in Section~\ref{sec:gen}. Wall-crossings are studied in Section~\ref{sec:wc}. Finally, we also look at applications to flops and the Bondal-Orlov conjecture, and prove the derived equivalence result for the Mukai flop in Section~\ref{sec:flops}.

\subsection{Windows over affine space}\label{sec:aff}
Here, we consider the case where the underlying scheme of the dg-scheme $ \X $ is affine space.  On the level of the rings, this means that we will choose (as in Example~\ref{ex:affineT},~\ref{ex:3.5})
\[
 T = k[\mathbf{x},\mathbf{y}],
\] where we use the shorthand notation $ \mathbf x  $ to mean $ x_1,x_2,\cdots,x_l $ and $   \mathbf y   $ to mean $ y_1,y_2,\cdots,y_m $ with internal $ \mathbb Z $-grading $ \dg x_i = a_i > 0 $ and $ \dg y_i = b_i < 0  $. We will repeatedly use the above convention for bold letters for notational convenience.   We also define the following two integers:
\[
\mu_+ := \sum a_i, \qquad \mu_{-} := \sum b_i.
\]

	\begin{lem}\label{lem:QPaffine} When $ T = k[\mathbf{x},\mathbf{y}] $ with the notations above,  the object $ Q(R) $ satisfies property $ P $, \ie the  morphism  \[ \rho : \left(Q(R)   \sopl Q(R)\right)_0 \to  Q(R)	,		\] from equation~\eqref{eq:rho} is an isomorphism.
\end{lem}
\begin{proof}
	In order to compute the derived tensor product, we need to find a semi-free resolution of $ Q(R) $. We claim that the following cdga $ K $ is a semi-free resolution of $ Q(R) $
	\[
	K := (R\otimes R[u, \boldsymbol\kappa,\boldsymbol\lambda,\boldsymbol\mu,\boldsymbol\nu], d),
	\] such that
	\[
	d \boldsymbol\kappa = (\mathbf{x}^2 - u^\mathbf{\dg \mathbf x}\mathbf{x}^1), \, \, d \boldsymbol\lambda = (  \mathbf{e}^2 - u^{\dg\mathbf{e}}\mathbf{e}^1 ), \, \, d \boldsymbol\mu = (\mathbf{y}^1 - u^\mathbf{- \dg \mathbf y}\mathbf{y}^2), \, \, d \boldsymbol\nu = (\mathbf{f}^1 - u^\mathbf{- \dg \mathbf f}\mathbf{f}^2).
	\] 

	In order to make the notation compact, we are using $ \boldsymbol \kappa $ to denote a set of elements (as many as the number of  $ x_i $s); and similarly for $ \boldsymbol \lambda $, $ \boldsymbol \mu $ and $ \boldsymbol \nu $. The superscripts ($ 1 $ or $ 2 $) refer to which copy of $ R $ the variable belongs to; for example $ \mathbf x^1 $ denotes $  \mathbf x^1 \otimes 1 \in R \otimes R $ (we apologize for this unfortunate choice of notation).
	
	It is easy to see that the morphism
	\[
		K \to Q
	\]  sending $ \boldsymbol \kappa, \boldsymbol \lambda, \boldsymbol\mu  $ and $ \boldsymbol \nu $ to zero is a quasi-isomorphism.

	Now,	we explicitly calculate $ \left(Q(R) \sopl Q(R)\right) $:
	\begin{align*}
	\left(Q(R)  \sopl Q(R)\right) &= T\otimes T[\boldsymbol\kappa,\boldsymbol\lambda,\boldsymbol\mu,\boldsymbol\nu, u]\sop k[\mathbf x, \mathbf z, \mathbf e, \mathbf g, v]\\
	&\cong k[\mathbf x,\mathbf x^1,\mathbf e, \mathbf e^1, \mathbf z, \mathbf y^1, \mathbf g, \mathbf f^1, \boldsymbol\kappa,\boldsymbol\lambda,\boldsymbol\mu,\boldsymbol\nu, u, v],
	\end{align*} where the dg structure remains unchanged on all the elements  expect $ \kappa,\lambda,\mu,\nu $ on which it is given by
	\[
		d \boldsymbol  \kappa = (\mathbf x - u^{\dg \mathbf x}\mathbf x^1), \, \,  d \boldsymbol  \lambda = (  \mathbf{e} - u^{\dg\mathbf{e}}\mathbf{e}^1 ), \, \, d \boldsymbol  \mu = (\mathbf{y}^1 - (uv)^\mathbf{- \dg \mathbf y}\mathbf{z}), \, \, d \boldsymbol  \nu = (\mathbf{f}^1 - (uv)^\mathbf{- \dg \mathbf f}\mathbf{g}).
	\]
	
	Now, we see that this cdga is quasi-isomorphic (by ``solving out" the above relations) to 
	
	 \[
	  k[\mathbf x^1, \mathbf e^1, \mathbf z, \mathbf g, u, v],
	 \] where the tri-degree of the generators are as follows
	 \[
	 \operatorname{tri}-\dg x^1_i = (\dg x_i, 0,0) \,\,  \operatorname{tri}-\dg e^1_i = (\dg e_i, 0,0) \,\, \operatorname{tri}-\dg z_i = (0, 0,\dg y_i),
	 \]
	 \[
	 \operatorname{tri}-\dg g_i = (0, 0,\dg f_i), \,\, \operatorname{tri}-\dg u = (-1,1, 0),\,\,\operatorname{tri}-\dg v = (0, -1,1).
	 \]

Taking middle degree $ 0 $ invariants we get
	\begin{align*}
	\left(Q(R)  \sopl Q(R)\right)_0 &\cong k[\mathbf x^1, \mathbf e^1, \mathbf z, \mathbf g, uv] \\
	&\cong Q(R).
	\end{align*}
	If we carefully trace through the series of isomorphisms in the proof, we see that it is induced by the morphism $ \rho $.
\end{proof}

\begin{lem}\label{lem:generators}
	
	The category of perfect objects $ \operatorname{Perf}^{\mathbb G_m}(\X^+)  $ is generated by $ j^*\mathcal{R}(i) $ for $ i \in (-\mu_+,0] $. 
	
\end{lem}

\begin{proof}

	The structure sheaf $ \mathcal{O}_{X^+} $ is an ample line bundle as $ X^+ $ is quasi-affine. By Definition~\ref{def:perf}, the category is generated by $ j^*\mathcal{R}(i)  $ where $ i \in \mathbb{Z} $. In fact, we can restrict to a smaller set of generators, $ j^*\mathcal{R}(i) $ where $  i \in (-\mu_{+},0] $ due to the following argument.
	
	Consider the Koszul complex $ \mathcal{K}_\bullet(x_1,x_2, \cdots, x_n) $ on $ X $
	\[
	\begin{tikzcd}[ column sep=large]
	0 \arrow[r, ] & \mathcal{O}_X(-\mu_{+}) \arrow[r, ] & \cdots \arrow[r, ] & \bigoplus_i \mathcal{O}_X(-a_i).\arrow[r,"{(x_1,x_2,\cdots,x_n)}"] & \mathcal{O}_X \arrow[r, ] &  0 \\
	\end{tikzcd}
	\]
	This is  acyclic when restricted to  $ X^+ $ and the twists in all intermediate terms lie in $(-\mu_{+},0]$. Now, we may apply  $ \otimes_{\mathcal{O}_X^+} j^*\mathcal{R} $ to see that $ j^*\mathcal{R}(-\mu_{+}) $  is generated as claimed. We can twist the Koszul complex by $ \mathcal{R}(p)  $ where $ p < 0 $ to show that $ j^*\mathcal{R}(i) $ for $ i \leq -\mu_{+} $ is generated as claimed. Similarly, we twist by $ \mathcal{R}(q) $ where $ q > 0 $, to show that $ j^*\mathcal{R}(i) $ for $ i > 0 $ is generated as claimed.
\end{proof}

Now, we would like to find a set of generators for the essential image of  $ \Phi_{Q_+}  $. For this result we need to make the assumption that all the positive generators $ e_i $ vanish, \ie $ R $ is semi-free over $ T $, where all the (homological) generators $ f_i $ of $R$ over $T$ are of non-positive internal degree.

 For the next lemma, we restrict the functor $ \Phi_{Q_{+}} $ to the subcategory of perfect objects $ \operatorname{Perf}^{\mathbb G_m}(\X^+) $, and show  that it takes perfect objects to perfect objects (and hence, \textit{a posteriori}, the notation $ \Phi_{Q_+} :  \operatorname{Perf}^{\mathbb G_m}(\X^+) \longrightarrow \operatorname{Perf}^{\mathbb G_m}(\X) $ is justified).

\begin{lem}\label{lem:window}  Assume that the degrees of the generators $ f_i $ of $ T \to R $ are non-positive, i.e., $ \operatorname{deg} f_i \leq 0 $.  Then, the image of $ \Phi_{Q_+} :  \operatorname{Perf}^{\mathbb G_m}(\X^+) \longrightarrow \operatorname{Perf}^{\mathbb G_m}(\X) $ is the full subcategory of $  \operatorname{Perf}^{\mathbb G_m}(\X) $ generated by $ \mathcal{R} (i) $ where $ i \in (-\mu_{+},0] $. 
		Moreover, we have  $$  \Phi_{Q_+} \circ j^* (\mathcal R (i) )  = \mathcal R(i)  \text{ where }  i \in (-\mu_{+},0]. $$
\end{lem}

\begin{proof}
	It suffices to find the image of $ \Phi_{Q_+} $ on the generators obtained in Lemma \ref{lem:generators} -- $ j^*\mathcal R(i) $ for $ i \in (-\mu_{+},0] $.
	 We have the following sequence of isomorphisms.
	\begin{align*}
	\Phi_{Q_+} (j^*\mathcal{R}(i))  &=  (R\pi_{2*}(Q_+ \; {}_p\otimes \pi_1^* j^* \mathcal{R}(i)))^{\mathbb{G}_m} \\
	&= (R\pi_{2*}Q_+ (i, 0))_{(0,*)} \\
	&=  (\pi_{2*}(Q_+ \otimes \mathcal C_{X^+ \times X})(i,0 ))_{(0,*)},
	\end{align*}where $ \mathcal C_{X^+ \times X} $ denotes the \v{C}ech resolution obtained by inverting the $ x_i $. Recall that the positively graded elements in $ T $ are $ x_1,\cdots, x_l $, with internal degrees $ a_1,\cdots,a_l $ respectively. We claim that we can use the \v{C}ech resolution to compute the derived pushforward $ R \pi_{2*} $ due to the following argument.
	
	The morphism $ \pi_2 : \X^+ \times \X \to \X $ factors as 
	\[\begin{tikzcd}
		\X^+ \times \X  \arrow[r,"for"] &(X^+ \times X, \mathcal O_{X^+} \boxtimes \mathcal R)  \arrow[r,"\pi^0"] & (X, \mathcal R) = \X,
		\end{tikzcd}
	\] where the first morphism is the forgetful morphism. To compute $ R\pi^0_* $, we can use the \v{C}ech resolution and then to compute $ R \pi_{2*} $, we  use the composition of derived functors.

Notice that we are tensoring over the $ p $ module structure and hence the result of the cohomology computation will have the $ s $ module structure.

We claim that the morphism 
\begin{equation}\label{eq:cech}
 \pi_{2*}(Q_+^b(i,0)) \to \pi_{2*}(Q_+^b \otimes \mathcal C_{X^+\times X}(i,0)) .
\end{equation}
 is an isomorphism for all $ b \leq 0 $ after taking internal degree $ (0,*) $ invariants. Here, the upper index $ b $ denotes the homological degree. First, let us compute the right hand side using a standard \v{C}ech cohomology computation. Notice that
\[
H^c(\pi_{2*}(Q_+^b \otimes \mathcal C_{X^+\times X}(i,0))) =  0 \qquad c \neq 0,l.
\] The $ c = l $ term is generated by
 \[
 \prod_{i=1}^l x_i^{-p_i} k[\mathbf z,  u] \prod_{\sum \operatorname{hdeg}f_i = b }g_i,
\] where $ p_i > 0 $ for all $ 0 \leq i \leq l $. Here, we use $ \operatorname{hdeg}  $ to denote the homological degree. Upon taking degree $ (i,*)  $ invariants, this term vanishes as $ i \in (-\mu_+,0] $, and $ \sum a_i = \mu_+ $. The left hand side of \eqref{eq:cech} is just a module and by taking its cohomology, we just mean taking degree $0$ invariants.  This can be computed directly as
\[
 \pi_{2*}(Q^b)_{(i,*)}  \cong u^{-i}	k[u^{\dg \mathbf x}. \mathbf x, \mathbf z] \prod_{\sum \operatorname{hdeg}f_i = b } g_i \cong s(\mathcal R)^b(i) 
\]

 Finally, we apply $ H^0 $ and $ (0,*) $ invariants to equation~\eqref{eq:cech} to  get
\[
  s(\mathcal R)^b(i) \to  H^0(\pi_{2*}(Q_+^b \otimes \mathcal C_{X^+\times X}))_{(i,*)} \cong  u^{-i}	k[u^{\dg \mathbf x}. \mathbf x, \mathbf z] \ \prod_{\sum \operatorname{hdeg}f_i = b } g_i ,
\] 
which is clearly an isomorphism and this proves our claim.

 Now, this claim implies that
\[
 \pi_{2*}(Q_+ \otimes \mathcal C_{X^+\times X}))_{(i,*)} \cong  \mathcal R(i),
\] with the $ s $-module structure and hence we are done.

\end{proof}

We can do the same analysis for $ \Phi_{Q_-} : \
\operatorname{Perf}^{\mathbb G_m}(\X^-) \longrightarrow \operatorname{Perf}^{\mathbb G_m}(\X)   $ where $ Q_- := Q|_{X^- \times X} $. Again, we need to make the assumption that all the negative generators $ f_i $ vanish, \ie $ R $ is semi-free over $ T $, where all the generators $ e_i $ are of non-negative degree.

\begin{lem}Assuming that the degrees of the generators $ e_i $ of $ T \to R $ are non-negative, i.e., $ \operatorname{deg} e_i \geq 0 $, the image of $ \Phi_{Q_-} :  \operatorname{Perf}^{\mathbb G_m}(\X^-) \longrightarrow \operatorname{Perf}^{\mathbb G_m}(\X) $ is the full subcategory of $ \operatorname{Perf}^{\mathbb G_m}(\X) $ generated by $ \mathcal R (i) $ where $ i \in [0, - \mu_{-}) $.
\end{lem}
\begin{proof}
	This holds by symmetry.
\end{proof}

Now, we look at an example of a hypersurface in affine space, in order to compare our conditions on the degree of $ e_i $ with Property $ (L+) $ and Property $  (A)   $ of \cite{HL}. In the specific case of $ \mathbb G_m  $-actions,  let us briefly recall the setting of \textit{loc.cit.},  and the aforementioned properties $ (A) $ and $ (L+) $. Consider a linearized $ \mathbb G_m $-action on a quasi-projective variety $ Y $, and denote the unstable locus by $ Y^{\operatorname{us}} $ and the fixed locus by $ Z $. Consider the following morphisms,
\[\begin{tikzcd}
Z \arrow[r,"\alpha", shift left]  & Y^{\operatorname{us}} \arrow[l, "\gamma", shift left] \arrow[r,"\beta"] & Y,
\end{tikzcd}
\] where $ \alpha  $, and $ \beta $ are the inclusions and $ \gamma $ is the projection. In this setup, Property $ (A) $ states that $ \gamma $ is a locally trivial bundle of affine spaces, and Property $ (L+) $ states that the derived restriction of the relative cotangent complex $ \alpha^*\mathbb L^\bullet_{Y^{\operatorname{us}}/Y} $ has non-negative weights.                           

\begin{exa}\label{ex:degreef}
	We consider the hypersurface case  of Example~\ref{ex:3.5}.  This is also studied in \cite[Example 2.5]{HL}. However, as a warning to the reader, the convention on weights in \textit{loc.\ cit.}\ appears to be the opposite of our work.
	Consider the cdga $ R = k[\mathbf x, \mathbf y, e] $, where $ d e = h $. The internal $ \mathbb Z $-grading is given by 
	\[
	\operatorname{deg} x_i >0,  \qquad \operatorname{deg} y_i < 0,
	\] as before. If we have the condition
	\[
	\dg h \leq 0,
	\] Lemma~\ref{lem:window} identifies the window subcategory explicitly.
	By comparison, Property $ (L+) $ and Property $ (A) $ of \cite{HL} can be satisfied either by imposing $ \dg h \geq 0 $ or by requiring that  $ h\!\mod \mathbf x $ is non-zero and linear in at least one of  the $ y_i $. Hence our condition on the degree appears to be almost complementary to the one imposed in \textit{loc.\ cit.}

\end{exa}

Now, let us  look at a specific case of Example~\ref{ex:degreef} with $\dg h \leq 0$.
\begin{exa}\label{ex:twopoints}
	We consider the cdga $ R = k[x_1,x_2,e] $ with $ de = x_1x_2 $ and internal $ \mathbb Z $-grading $ \dg x_1 = \dg x_2 =1 $.  Note that $ \dg e = 2 $ in this example, and hence Lemma~\ref{lem:window} does not apply. However, we wish to  understand the essential image of the  window functor $ \Phi_{Q_+} $ in this example. It is easy to do so by direct computation.  As we wish to compare our window to \cite{HL}, we also use the isomorphism 
	\begin{align}\label{eq:ex1}
	\Perf^{\mathbb G_m} (R) \cong \Perf^{\mathbb G_m} (k[x_1,x_2]/(x_1x_2)).
	\end{align}
		
	It suffices to find the image of $ j^* \mathcal R $ as $ j^* \mathcal R(-1) \cong j^* \mathcal R $. So, we compute $ \Phi_{Q_+}(j^* \mathcal R) $ as in Lemma~\ref{lem:window} to get
	\begin{align}\label{eq:ex}
	\Phi_{Q_+}(j^* \mathcal R) =( Q_+ \otimes \mathcal C_{X^+ \times X})_{(0,*)} = \left(k[x_1,x_2,x_1^{-1},e,u] \oplus k[x_1,x_2,x_2^{-1},e,u]\right)_{(0,*)}.
	\end{align} To get this, observe that the second (and only other) term of the \v{C}ech complex is 
	\[
	k[x_1,x_2,x_1^{-1}, x_2^{-1},e,u] \cong 0
	\] as $ de = x_1x_2 $ is inverted.
	
	Under the isomorphism~\ref{eq:ex1}, we can compute the right hand side of equation~\ref{eq:ex} to get
	\begin{align*}
	\left(k[x_1,x_2,x_1^{-1},u]/(x_1x_2) \oplus k[x_1,x_2,x_2^{-1},u]/(x_1x_2)\right)_{(0,*)} &\cong \left(k[x_1,x_1^{-1},u] \oplus k[x_2,x_2^{-1},u]\right)_{(0,*)}\\ & \cong k[ux_1] \oplus k[ux_2],
	\end{align*}  where we view the results of the calculation using the $ \sigma $-module structure. Hence we see that our window is generated by the objects $ k[x_1] $ and $ k[x_2] $. 
	
	\noindent Finally, one can check that the window subcategory as defined in \cite{HL} is generated by $ k[x_1] $ and $ k[x_2] $, and hence the windows coincide.

\end{exa}

\subsection{Windows in general}\label{sec:gen}

Now, we consider the general case (for us), where $ T $ is a finitely generated smooth $ \mathbb Z $-graded ring over a \textit{field} $ k $. We need to restrict to a field $ k $ in order to use the Luna slice theorem. We recall the following result from \cite{bdf} for (non-dg) $ k $-algebras.

\begin{lem}\label{lem:QPsmooth}
	For any smooth finitely generated $ \mathbb Z $-graded  $ k $-algebra  $ T  $, $ Q(T) $ satisfies Property $ P $ \ie
	\[
		\rho_T : \left(Q(T) \otimes^L Q(T)\right)_0 \cong Q(T)
	\]
\end{lem}
\begin{proof}
	The proof uses the Luna slice theorem to reduce to the case of affine space. See \cite[Lemma 4.2.5, Proposition 4.2.6]{bdf}.
\end{proof}

\begin{lem}\label{lem:QP} When $ T $ is a finitely generated smooth $ \mathbb Z $-graded ring over a field $ k $,  the object $ Q(R) $ satisfies property $ P $, \ie the  morphism  \[ \rho : \left(Q(R)   \sopl Q(R)\right)_0 \to  Q(R)	,		\] is an isomorphism.
\end{lem}

\begin{proof}

	 	Consider a semi-free resolution $ K $ of $ Q(T) $ as a $ T \otimes T $-module. This provides a semi-free resolution $ K[\mathbf e, \mathbf g] $ of $ Q(R) $ and we can compute the derived tensor product,
	\begin{align*}
	Q(R) \sopl Q(R)&\cong \left(K[\mathbf e, \mathbf g] \sop  Q(T)[\mathbf e', \mathbf g']\right)\\
	&\cong K \otimes_{{}^sT^p} Q(T) [\mathbf e, \mathbf g']\\
	&\cong Q(T) \otimes_{{}^sT^p}^L Q(T) [\mathbf e, \mathbf g'].
	\end{align*} 

Hence
	\begin{align*}
	\left( Q(R) \sopl Q(R)\right)_0 & \cong \left(Q(T) \otimes_{{}^sT^p}^L Q(T)\right)_0 [\mathbf e, \mathbf g'] & \text{ since }\mathbf e, \mathbf g' \text{ have middle degree }0\\
	& \cong Q(T)[\mathbf e, \mathbf g'] & \text{ by Lemma~\ref{lem:QPsmooth}} \\
	& \cong Q(R). &
		\end{align*}		 
	 Tracing through the chain of maps, we see that the map inducing the isomorphism is $ \rho $.

\end{proof}

\subsubsection{Defining windows}
Let us recall the definition of windows in the  setting of ordinary schemes as in \cite{bdf} now. Let $ T $ be a $ \mathbb Z $-graded smooth ring over a field $ k $. Let $ \mu_\pm $ be the sum of the weights of the conormal bundle of $ \operatorname{Spec} T/J^\pm $ in $ \operatorname{Spec} T = X $, where $ J^\pm $ is the ideal generated by all the positive/negative-ly graded elements. For simplicity, we assume that the fixed locus $ V(J^+) \cap V(J^-) $ is connected. If not, we need to define different $ \mu_\pm $ for different connected components.

Then we define the grade restriction window $ \mathbb{W}^+_X  $ to be the full subcategory of $ \operatorname{D}^b(\operatorname{QCoh}^{\mathbb{G}_m}X) $ generated by objects  $ A  $ such that for any fixed point $ y \in  X $, and some affine \'etale neighborhood $ V = \operatorname{Spec} S $ in $ X = \operatorname{Spec} T $,  the restriction 
\[
A \otimes_{T} S 
\] is generated by $ S(i) $ for $ i \in (-\mu_{+},0] $.  Notice that this definition of the window matches the one from Lemma~\ref{lem:window} in the case where $ X $ is affine space.

In our setting of dg-schemes we will view the scheme $ X = \Spec T $ as a dg-scheme $ (X, \mathcal O_X) $, but by abuse of notation denote it by just $ X $. We use $ j' $ for the following inclusion:
\[
j' : (X^+,\mathcal{O}_{X^+}) \hookrightarrow (X,\mathcal O_X).
\]

For the dg-scheme  $ \mathbf X = (X, \mathcal R) $ that we are interested in, we will define the window $ \mathbb{W}^+_{\mathbf X} $ as the pullback of the window $ \mathbb{W}^+_X $ under the forgetful morphism. Just as in the case of ordinary schemes, we assume that the fixed locus is connected for simplicity.  
In more detail, we have the following.
\begin{defin} \label{defin: window}
Let
\[
f : \X = (X,\mathcal R) \to X = (X, \mathcal O_X) 
\]
 be the forgetful morphism.
The \emph{grade restriction window}
 $ \mathbb{W}^+_{\mathbf X} $ to be the full subcategory of $ \operatorname{Perf}^{\mathbb{G}_m}(\mathbf X) $ generated by objects $ f^*A $ with $A \in \Perf^{\mathbb G_m}(X)$ such that for any fixed point $ y \in  X $, and some affine \'etale neighborhood $ V = \operatorname{Spec} S $ in $ X = \operatorname{Spec} T $,  the restriction 
\[
A \otimes_{T} S 
\] is generated by $ S(i) $ for $ i \in (-\mu_{+},0] $  i.e.\
\[
\mathbb{W}^+_\mathbb{\mathbf X} := \langle f^*A ,  A \text{ in } \mathbb{W}^+_X \rangle.
\]

 \end{defin}
 
 We also introduce the forgetful morphism on the semi-stable locus
 \[
 f^+ : \X^+ \to X^+, 
 \] and the forgetful morphism
 \[
 \alpha := f^+ \times f : \X^+ \times \X \to X^+ \times X
 \]

Let us reiterate that we will denote the dg-scheme $ (X,\mathcal O_X) $  by just $ X $, as opposed to the dg-scheme $ \X = (X,\mathcal R) $. The following commutative diagrams summarize the notation:
\[\begin{tikzcd}
& \X^+ \times \X \arrow[r,  "\alpha" ] \arrow[ld,"{\pi_{1}}", swap] \arrow[rd,"\pi_{1}'", near start, swap ] & X^+ \times X \arrow[ld,"{\pi}_{2}", near start, ] \arrow[rd,"{\pi_{2}'}",  ] & \\
\X^+ \arrow[r, "f^{+}" ] & X^+ & \X \arrow[r,"f"]& X, 
\end{tikzcd}
\]
and 
\[\begin{tikzcd}
\X^+ \arrow[r, hook, "j"] \arrow[d, "f^{+}" ]& \X \arrow[d,"f"] \\
X^+ \arrow[r,hook,"j'"]& X.
\end{tikzcd}
\]

The following result about base change  for $ Q $ under the forgetful morphism $ f : \X \to X $ will help us compute the window $ \mathbb W_\X^+ $.

\begin{lem}\label{lem:Qfcomm}
	Assuming that the degrees of the generators $ f_i $ of $ T \to R $ are non-positive, i.e., $ \operatorname{deg} f_i \leq 0 $, the following is an isomorphism of functors
	\[
	\Phi_{Q(R)} \circ f^* \cong f^* \circ \Phi_{Q(T)},
	\] where $ Q(R) $ and $ Q(T) $ are considered with the $ p \otimes s $-module structure.
\end{lem}

\begin{proof}

	We claim that the  following morphism of $ \mathbb Z \times \mathbb Z $-graded $ R\otimes R $-dg-modules
	\begin{equation}
	\begin{aligned}[c]
	 Q(T) { }_s\!\otimes_T R &\to Q(T)[\mathbf g] \\
	q \otimes 1 &\mapsto q \\
	1\otimes f_i &\mapsto g_i \\
	1 \otimes t &\mapsto s(t),
	\end{aligned}
	\end{equation} where $ q \in Q(T), f_i \in R $ and $ t \in T \subset R $, is an isomorphism.

	If we forget the dg-structure, it is clear that the map is surjective and injective, and respects the grading when we view $ R  $ in $ Q(T) { }_s\!\otimes_T R $ with the grading $ (0,*) $. The morphism also respects the  dg-structure; if $ d f_i $ is in $ T $,
	\[
	d (1 \otimes f_i)  = 1 \otimes df_i = s(df_i) \otimes 1 \mapsto s(df_i) = dg_i.
	\] This proves the isomorphism
	\[
	Q(T) _s\!\otimes_T R \cong Q(R).
	\]
	
	Now, let $ M  $ be an object of $ D(X) $. Then
	\begin{align*}
		\Phi_{Q(R)} \circ f^* (M) &= \pi_{2*}( (M \otimes_T R) \otimes_{R} {}_pQ(R))^{\mathbb G_m} \\
		&=  \pi_{2*}((M \otimes_T R) \otimes_R {}_p(Q(T)_s\!\otimes_T R))^{\mathbb G_m} \\
		&= \pi_{2*}( (M \otimes_T {}_pQ(T)_s\!)\otimes_T R)^{\mathbb G_m} \\
		&= \pi_{2*}( (M \otimes_T {}_pQ(T)))^{\mathbb G_m} {}_s\!\otimes_T R \\
		&= f^* \circ \Phi_{Q(T)}(M)
	\end{align*} To get the second to last line of the above chain of isomorphisms, we use that $ R $ has grading $ (0, *) $ and that the  $ \mathbb G_m $-invariants is with respect to the first grading.
\end{proof}

We recall one of the main results of \cite{bdf}.
\begin{lem}[\!\!{\cite[Theorem 4.2.9]{bdf}}] \label{lem:bdfequiv} When $ X $ is a smooth scheme over a field $ k $, the following  
	\[\begin{tikzcd}
	\Phi_{Q(T)_+} : D^b(X\sslash{+})  \arrow[r, shift right] & \arrow[l, shift right] \mathbb{W}^+_X : {j'}^*
	\end{tikzcd}	
	\] is an equivalence of categories. 
\end{lem}

Finally, we show that our window functor $ \Phi_{Q_+} $ lands in the window $ \mathbb{W}^+_{\mathbf X} $. 

\begin{lem}\label{lem:windowgen}Assuming that the degrees of the generators $ f_i $ of $R$ as a $T$-algebra are non-positive, i.e., $ \operatorname{deg} f_i \leq 0 $, the image of $ \Phi_{Q_+} :  \operatorname{Perf}^{\mathbb G_m}(\X^+) \longrightarrow \operatorname{Perf}^{\mathbb G_m}( \mathbf X) $ lies in the window subcategory $ \mathbb{W}^+_{\X} \subset \operatorname{Perf}^{\mathbb G_m}(\X) $.

\end{lem}

\begin{proof}

	We have the following isomorphisms:
	\begin{align*}
	f^* \circ \Phi_{Q(T)} \circ j'_*&\cong \Phi_{Q(R)} \circ f^* \circ j'_* \\
	&\cong  \Phi_{Q(R)} \circ j_* \circ f^{+*} \\
	&\cong \Phi_{Q(R)_+} \circ f^{+*},
	\end{align*} where we used the isomorphism of Lemma~\ref{lem:Qfcomm} in the first line. The second line follows using the projection formula,
	\begin{align*}
	j_* \circ f^{+*} (\mathcal F)&= j_*(j'^*R \otimes \mathcal F) \\
	&=  R \otimes j'_* \mathcal F \\ 
	&= f^* \circ j'_* (\mathcal F) .
	\end{align*}

	This gives us  the following isomorphism of functors.
	\[
	\Phi_{Q(R)_+} \circ f^{+*} \cong f^* \circ \Phi_{Q(T)_+}.
	\]
	
Clearly,  $ \operatorname{Perf}^{\mathbb G_m}(\X^+) $ is generated by the essential image of $ {f^{+}}^{*} $. Hence we only need to find the image of $ \Phi_{Q_+} $ on objects of the form $ {f^+}^*{j'}^* A $, where $ A $ is in $ \mathbb{W}^+_X $.

	Now,  if we consider an object $ A  $ in $ \mathbb W^+_X $,
	\begin{align*}
	\Phi_{Q(R)_+}(f^{+*}j'^*A) &\cong f^* \circ \Phi_{Q(T)_+}( j'^* A) \\
	&\cong f^* A,
	\end{align*} where we used Lemma~\ref{lem:bdfequiv} in the last line.
	
	This shows that the essential  image of $ \Phi_{Q(R)_+} $ lands in the window $ \mathbb W^+_\X $.
\end{proof}

The previous statement can  now be upgraded to the following theorem.

\begin{thm}\label{thm:Qequiv}
Assume that $R$ is generated as a $T$-algebra by non-positive elements. The functor
	\[
	\Phi_{Q_+}  : \operatorname{Perf}^{\mathbb G_m}(\X^+) \longrightarrow \mathbb{W}^+_{\X}
	\] is an equivalence of categories with inverse functor $ j^*  $.
\end{thm}
\begin{proof}
	The fullness follows from Lemma~\ref{lem:FF} using Lemma~\ref{lem:QP}.  The faithfulness was proved in Lemma~\ref{lem:faithfulness}. Lemma~\ref{lem:windowgen} implies that a set of generators of $ \mathbb{W}^+_{\X} $ lies in the essential image of $ \Phi_{Q_+}$. This shows that $ \Phi_{Q_+}  $ is an equivalence.
	The inverse functor is $ j^* $ by the proof of Lemma~\ref{lem:windowgen}.
\end{proof}

Similarly, we can define the negative window functor and the negative window for the negative GIT quotient, when the generators $ e_i $ of $ T \to R $ are non-negative.

\begin{thm}\label{thm:Qequivn}
Assume that $R$ is generated as a $T$-algebra by non-negative elements.  The functor
	\[
	\Phi_{Q_-}  : \operatorname{Perf}^{\mathbb G_m}(\X^-) \longrightarrow \mathbb{W}^-_{\X}
	\] is an equivalence of categories with inverse functor $ j_{-}^*  $, where $ j_{-} : \X^- \hookrightarrow \X $ is the inclusion of dg-schemes.
\end{thm}
\begin{proof}
This is true by symmetry.
\end{proof}

\begin{rem}\label{rem:Ldeg} A more invariant way to describe the condition that $ R $ is generated by non-positive (or analogously, non-negative) elements over $ T $ is as follows.  As $ R $ is semi-free over $ T $,  we can compute the relative cotangent complex $ \mathbb L_{R/T}  $ as 
\[
	\mathbb L_{R/T} = \Omega_{R/T}.
\] Now, we can restrict this cotangent complex to the fixed locus of $ X $ (which is connected by assumption)  and consider the weights of this complex. As all the differentials are zero upon restriction, the  generators of $ R $ over $ T $ have non-positive degree if and only if the restriction of the cotangent complex to the fixed locus  $	\gamma^* \mathbb L_{R/T} $ has non-positive degree. Here  $ \gamma $ denotes the  inclusion of the  fixed locus into the scheme $ X $.

Notice that in the case where $ R $ is a semi-free  resolution of a quotient ring $ T/I $ (for example, the setting of flops as in Section~\ref{sec:flops}), this condition can be rephrased as requiring that the cotangent complex $  \mathbb L_{\left(T/I\right) /T} $ restricted to the fixed locus of $ \Spec T/I $ has non-positive weights. However, note that the condition is dependent on the presentation of $ T/I $ as a quotient of $ T $ (i.e., the closed immersion $ \Spec T/I \to \Spec T $).
\end{rem}

\subsection{Wall-crossings}\label{sec:wc}

 By combining Theorem~\ref{thm:Qequiv} and Theorem~\ref{thm:Qequivn}, we can prove results about the wall-crossing functors when all the \textit{generators $ e_i $ of $  T \to R $ have degree $ 0 $}. Hence, we assume that the generators $ e_i $ of the semi-free cdga $  R $ over $ T $ have internal degree $ 0 $ in this section.

We consider the easiest case first. Recall that $ T $ is a $ \mathbb Z $-graded smooth ring over $ k $, and that $ J^\pm $ are generated by the positively/negatively graded elements of $ T $. Let $ \mu_{\pm} $ be the sum of the weights of the conormal bundle of $ \Spec T/J^{\pm} $ in $ \Spec T $. Our first result is when $ \mu_+ +\mu_- = 0 $, often referred to as the Calabi-Yau condition.

\begin{thm}\label{thm:wcequiv} Let $ R  $ be a semi-free cdga equipped with a $ \mathbb{G}_m $-action such that $ T = R^0 $ is smooth, and the generators  $ e_i $ of $ R $ over $ T $ have internal degree $ 0 $. Let $ \mu_{\pm} $ be the sum of the weights of the conormal bundle of $ T/J^{\pm} $ in $ T $. Let $ j_{-} : \X^{-} \hookrightarrow \X $ be the inclusion of dg-schemes.
	When, $ \mu_{+} + \mu_{-} = 0 $, the wall crossing functor
	\[
	\Phi^{\operatorname{wc}} := j_{-}^* \circ \left(-\otimes\mathcal{R}(\mu_+-1)\right)\circ \Phi_{Q_+} : \operatorname{Perf}^{\mathbb G_m}(\X^+) \longrightarrow \operatorname{Perf}^{\mathbb G_m}(\X^-)
	\] is an equivalence of categories.
\end{thm}
\begin{proof}
	Theorem~\ref{thm:Qequiv} shows that $ \Phi_{Q_+} $ gives an equivalence of categories
	\[
	\operatorname{Perf}^{\mathbb G_m}(\X^+) \cong \mathbb{W}^+_\X.
	\]	
	A similar analysis for $ \Phi_{Q_-} $ shows that 
	\[
	\operatorname{Perf}^{\mathbb G_m}(\X^-) \cong \mathbb{W}^-_\X.
	\] We have the condition $ \mu_+ +\mu_{-} = 0 $ which ensures that applying  $ \left(-\otimes\mathcal{R}(\mu_+-1)\right) $ exchanges the positive and negative windows 
	\[
	\mathbb{W}_\X^+ \otimes \mathcal{R}(\mu_+ -1) \cong \mathbb{W}_\X^-.
	\] As $ j_{-}^* $ is the inverse functor to $  \Phi_{Q_-}  $, we have the result.
\end{proof}

Now, we would like to understand the case when the windows are of different lengths, \ie when $ \mu_+ + \mu_- \neq 0 $.  We only consider the (easier) case when  $ X $ is an affine space.

\subsubsection{Affine space}
Recall the setting of Section~\ref{sec:aff}, where
 \[
T = k[\mathbf{x},\mathbf{y}],
\] with internal $ \mathbb Z $-grading $ \dg x_i = a_i > 0 $ and $ \dg y_i = b_i < 0  $.  The assumption that the generators  $ e_i $ of $ R =  T[\mathbf e] $ are of internal degree zero still stands.

For convenience, we also introduce the notation
\[
\mathbb W_{[a,b]} = \langle R(a) , R(a+1) ,\cdots , R(b) \rangle \subset \Perf^{\mathbb G_m}(\X), \qquad a < b  \text{ in } \mathbb Z .
\]

Let us consider the case when $ \mu_+ +\mu_- > 0 $, i.e, the positive window $ \mathbb W_\X^+ $ is longer than the negative window $ \mathbb W_\X^- $. Then, we have the following semi-orthogonal decomposition for the window.

\begin{lem}\label{lem:SOD1}
Assume $b-a  \geq  \mu_+$.	The following is a semi-orthogonal decomposition
	\[
	\mathbb W_{[a,b+1]} = \langle  \Perf(R^{\mathbb G_m}),  \mathbb  W_{[a,b]}\rangle.
\]
where $R^{\mathbb G_m} := R/ (\mathbf x, \mathbf y)$ is the dg fixed locus.  Furthermore, $\Perf(R^{\mathbb G_m})$ is the category generated by $R/\mathbf x (b+1)$.
\end{lem}
\begin{proof}
	First, we show that the category $ W_{[a,b+1]} $ is generated by $  R/\mathbf x (b+1) $ and  $   W_{[a,b]} $. The only missing generator is $ R(b + 1) $, and we can get this by considering the  following Koszul resolution
		\[
		 R \otimes \mathcal{K}^\bullet (x_1,\cdots, x_l)(b+1) \cong R/\mathbf x(b+1).
		\] The top term of the resolution is $ R(b + 1) $ and the rest of the terms are in $W_{[a,b]}$.  Hence $ R(b + 1) $ is generated by $ \langle  R/\mathbf x (b+1),   W_{[a,b]}  \rangle $.
		
	Now, it suffices to verify the semi-orthogonal decomposition condition on the generators. Hence we need to check that
		\[
		R\Hom( R(i),R/\mathbf x (b+1)) = 0 \qquad \text{ for }  a \leq i\leq b.
		\] We have  
		\[
	 	R\Hom( R(i),R/\mathbf x (b+ 1)) = (R/\mathbf x )_{(b+1-i)} =  k[\mathbf y]_{(b+1-i)}
		\] as $ b+1 - i > 0  $ and hence, we are done.
			
We now compute the dg-endomorphism ring of $R/\mathbf x(i)$ as:
		\begin{align*}
		R\Hom( R/\mathbf x,R/\mathbf x) &= R\Hom (R \otimes \mathcal K(x_1,\cdots, x_l) , R/\mathbf x) \\
		&= \left( \bigwedge \!\!{}^\bullet \bigoplus_{i=1}^i k(a_i) \otimes R/\mathbf x \right)_{0} \\
		&=  R^{\mathbb G_m}
				\end{align*}

				If the category $  \langle  R/\mathbf x (b+1) \rangle $ is idempotent-complete, a result of Keller \cite[Theorem 3.8b)]{Keller} says that  we can identify this category $ \langle  R/\mathbf x (b+1) \rangle $ with the  derived category of perfect dg-modules of the endomorphism ring of the generator $ R/\mathbf x (b+1) $. Now, we note that $ \langle R/\mathbf x (b+1) \rangle $ is idempotent-complete as it is generated by a compact object of the category $ D(\X) $ which is idempotent-complete as it admits countable co-products, and hence we have the equivalence
				\[
				 \langle  R/\mathbf x (b+1) \rangle  \overset{\sim}{\to} \Perf(R^{\mathbb G_m}).
				\]
\end{proof}

Using this, we can get a semi orthogonal decomposition for $ \Perf^{\mathbb G_m}(\X^+) $ when $ \mu_+ +\mu_- > 0  $.

\begin{thm}\label{thm:wcsod}
	The following is a semi-orthogonal decomposition
	\[
	\Perf^{\mathbb G_m}(\X^+) = \langle \Perf(R^{\mathbb G_m})_{\mu_++\mu_-},\cdots, \Perf(R^{\mathbb G_m})_1, \Phi^{wc}( \Perf^{\mathbb G_m}(\X^+))\rangle
	\]
	where $\Perf(R^{\mathbb G_m})_{i} \cong \Perf(R^{\mathbb G_m})$ denotes the full subcategory generated by  $ j_{+}^*R/\mathbf x (i)$.
\end{thm}

\begin{proof}
	First, we use Lemma~\ref{lem:SOD1} inductively to get the  semi-orthogonal decomposition 
		\[
		W_{[-\mu_++1,0]} = \langle R/\mathbf x (\mu_++\mu_- ),\cdots, R/\mathbf x (2), R/\mathbf x (1),   W_{[\mu_- +1,0]}\rangle,
		\] in $ \Perf^{\mathbb G_m}(\X) $.	Then, we pull back along $ j_{+} : \X^+ \hookrightarrow \X $ to get the result.
 	
\end{proof}

Finally, we  also describe the kernel for the wall-crossing functor. We need to define the affine GIT quotient $ \X\sslash{0} $ (with stability corresponding to the linearization with weight $ 0 $). Define the cdga $ R_{(0)} $  (this is not the same as $ R^0 = T $!) and the dg-scheme $ \X\sslash{0} $ as 
\[
R_{(0)} := k[\mathbf x, \mathbf y, \mathbf e]^{\mathbb G_m}, \qquad \X\sslash{\,0} := (\Spec T, R_{(0)} ).
\] Then, the wall-crossing kernel (viewed as a dg-scheme) is finite over $ \X\sslash{+}\times_{\X\sslash{0}} \X\sslash{-} $.

\begin{prop}\label{prop:fiber} In $ D^{\mathbb G_m}(\X^+ \times \X^-) $, the natural morphism, obtained by restricting the $ R \otimes R $-module structure on $ Q $,
	\[
	\mathcal R^+ \otimes_{R_{(0)}} \mathcal R^- \longrightarrow Q|_{\X^+ \times \X^-},
	\] is finite.  
	
	 Moreover, if we  assume that $ a_i = 1 $ for all $ i = 1,\cdots, l$, and that $ b_j = -1 $ for all $ j = 1,\cdots, m $,  the above morphism is an isomorphism.
\end{prop}
\begin{proof}
	We get the morphism in the statement of the proposition by restricting the $ R \otimes R $-module structure morphism of $ Q $ to  $ \X^+ \times \X^- $,
	\[
	\mathcal R^+ \otimes_{k} \mathcal R^- {\longrightarrow} Q|_{\X^+ \times \X^-}.
	\] However, even before restriction, we see that if $ r $ is an element of internal degree zero in $ R $, $ r \otimes 1 - 1 \otimes r $ goes to $ 0 $ in $ Q $. Hence, we get the map
	\[
	\mathcal R^+ \otimes_{R_{(0)}} \mathcal R^- \overset{f}{\longrightarrow} Q|_{\X^+ \times \X^-},
	\] which we  call $ f $. 
	
In fact, we show that the (less restricted) map 
	\[
	 \mathcal R^+ \otimes_{R_{(0)}} \mathcal R \overset{f}{\longrightarrow} Q|_{\X^+ \times \X},
	\] 
	is finite.  This is verified on an equivariant open cover of $ \X^+ \times \X$.  
Recall that $ Q = k[\mathbf x, \mathbf z, u, \mathbf e] $, and that the morphism $ f $ sends $ \mathbf x \otimes 1 \mapsto \mathbf x $ and $ 1 \otimes \mathbf y \mapsto \mathbf z$.

 On the chart, we get $ u^{\dg x_a} $ as
	\[
	f(x_a^{-1} \otimes x_a) = u^{\dg x_a}.
	\] 
	Hence, $Q|_{U_{x_a} \times \X}$ is (finitely) generated by $1, u, ..., u^{\dg x_a -1}$ as a $R[x_a^{-1}] \otimes_{R_{(0)}} \mathcal R$-module.

Now assume that $ a_i = 1 $ and $ b_j = -1 $. Then, on the open chart $ U_{x_a} \times U_{y_b} $, we have the isomorphism
	\begin{equation}\label{eq:iso}
	 R[x_a^{-1}] \otimes_{R_{(0)}} R[y_b^{-1}] \cong \left(k[\mathbf x, x_a^{-1}] \otimes_k  k[\mathbf y, y_b^{-1}, x_a] \right)[\mathbf e],
	\end{equation}as follows. We solve for $ y_j \otimes 1 $ when $ j = 1, \cdots, m $ and $ 1 \otimes x_i $ when $ i = 1, \cdots, a-1,a+1,\cdots, l $, using the relations
	\[
	y_j \otimes 1 = x_a^{-1} \otimes x_a y_j, \qquad 1 \otimes x_i = x_i y_b \otimes y_b^{-1}.
	\]  We can easily check that all the other relations are already satisfied, and this gives the isomorphism~\eqref{eq:iso}. Now, the map $ f $ induces an isomorphism between $ k[\mathbf x, x_a^{-1}] \otimes k[\mathbf y, y_b^{-1}, x_a] $ and $ Q|_{\X^+ \times \X^-} $ with inverse $ \mathbf x \mapsto \mathbf x \otimes 1,  \mathbf z \mapsto 1 \otimes \mathbf y, \mathbf e \mapsto \mathbf e, u \mapsto x_a^{-1} \otimes x_a$,  and we are done. 
\end{proof}

The situation here is analogous to  the Atiyah/standard flop where the kernel for the wall-crossing functor is the fiber product of the $ \pm $GIT quotients over the  $ 0 $-GIT quotients \cite{BO}. We  discuss some other aspects of  the wall-crossing functor in the context of flops in Section~\ref{sec:Qpm}.

\subsection{Applications to flops}\label{sec:flops}

As discussed in the introduction, one of the major applications of our results that we have in mind is to the Bondal-Orlov conjecture on derived equivalences of flops \cite{BO}. By an observation of Reid \cite[Proposition 1.7]{Tha}, a flop between smooth projective varieties can be realized as different GIT quotients of $ \mathbb G_m $ acting on a scheme $ Y $. Let us consider the local setting and assume that $ Y $ is affine. In general, $ Y $ is singular, and the idea is to resolve it as a semi-free cdga in order to construct the wall-crossing functor. We can prove the derived equivalence under certain conditions, using the results of the previous sections.

To be precise, consider a  sub-variety $ Y $ of a smooth affine variety $ X = \Spec T $, equipped with a $ \mathbb G_m $-action. Then our flop diagram is given as 
\[\begin{tikzcd}
Y\sslash{+} \arrow[dotted, rr] \arrow[rd]&  &  Y\sslash{-}\arrow[ld] \\
& Y\sslash{\,0} := \Spec(k[Y]^{\mathbb{G}_m}),&
\end{tikzcd} 
\] where the $ \pm $-GIT stability conditions are exactly the same as discussed in Section~\ref{sec:VGIT}. Now, consider the Koszul-Tate resolution of $ k[Y] = T/I $ to get a semi-free cdga $ R $ over $ T $; in particular, the cdga $ R $ can be chosen such that there are only finitely many generators in each homological degree \cite[Lemma 23.6.9]{stacks-project}. In general, we require an infinite number of dg-algebra generators. Then, we have
\[
R =( T[e_1,e_2,\cdots],d), \qquad \text{such that} \qquad H^0(R) \simeq k[Y].
\] 
 If $ Y $ is a complete intersection in $ X $ defined by functions $ (h_1,\cdots,h_n) $, we can take the semi-free cdga $ R $ to be the  Koszul resolution; in this case, the number of dg-algebra generators is finite, and the differential acts as
 \[
 d e_i = h_i, \qquad i \in [1,n].
 \]

In either case, if we assume that the internal  degree of all the $ e_i $ is zero,  we can use the results of the previous section, in particular Theorem~\ref{thm:wcequiv}, to obtain the following result. 

\begin{cor}\label{cor:ci}
	Assume that the condition $ \mu_+ + \mu_- = 0 $ is satisfied. Also assume that the dg-algebra generators $ e_i $ in $ R $  have internal degree zero.  Then the wall crossing functor	
	\[
	\Phi^{\operatorname{wc}} : \operatorname{Perf}(Y\sslash{+}) \longrightarrow \operatorname{Perf}(Y\sslash{-})
	\] is an equivalence of categories.
\end{cor}
\begin{proof}
	This is a direct application of Theorem~\ref{thm:wcequiv} to the setting of this section. We use the fact that quasi-isomorphisms of dg-schemes induce equivalences of derived categories, which is an immediate extension of \cite[Proposition 1.5.6]{Riche} to the equivariant setting, \ie
	\[
		D^{\mathbb G_m}(X^\pm,\mathcal R^\pm) \cong D^{\mathbb G_m}(Y^\pm,\mathcal O)
	\] as $ \mathcal R $ is a dg-resolution of $ k[Y] $. As mentioned in Remark~\ref{rem:perf}, it is easy to see that our ad-hoc definition of the category of perfect objects coincides with the standard one in the case of ordinary schemes, and hence we are done.
\end{proof}

Moreover, if we do not have the condition $ \mu_+ + \mu_- = 0 $, the results on wall crossings (Theorem~\ref{thm:wcsod}) can be carried over to our setting in this section. These statements are straightforward to write down, but we leave it to the interested reader.

\subsubsection{Mukai flop}

In particular, this allows us to give a VGIT proof of the derived equivalence for the Mukai flop. Let us briefly remind the reader about the VGIT construction of the Mukai flop. Here, $ k $ will be an arbitrary Noetherian ring, and we recall Example~\ref{ex:affineT}.  Consider the $ \mathbb Z $-graded ring $ T = k[x_1,\cdots, x_l, y_1,\cdots, y_l] $ with internal degree
\[
\operatorname{deg} x_i = 1 \qquad \operatorname{deg} y_i = -1.
\] Consider the cdga
\[
	R = T[e], \qquad d e = \sum_{i=1}^l x_i y_i,
\]  which is quasi-isomorphic to the ring
\[
	S:= k[x_1,\cdots, x_l, y_1,\cdots, y_l]/(\sum_{i=1}^l x_i y_i).
\] The (local model of the)  Mukai flop is the birational transformation between the two GIT quotients 
\[
\begin{tikzcd}
Y \sslash{+} \arrow[r, dotted] & \arrow[l, dotted] Y \sslash{-},
\end{tikzcd}
\] where $ Y = \Spec S $. We also define the invariant ring 
\[
S_{(0)} := S^{\mathbb G_m}
\] As the $ e $ has internal degree zero, we can apply Corollary~\ref{cor:ci} to get the derived equivalence.
\begin{cor}\label{cor:mf}
	Consider the VGIT presentation of the Mukai flop as above. The wall crossing functor:
	\[
	\Phi^{\operatorname{wc}} : D^b(Y\sslash{+}) \longrightarrow D^b(Y\sslash{-})
	\] is an equivalence. Moreover, the kernel for the equivalence is (quasi-isomorphic to) the structure sheaf of the fiber product:
	\[
	 \mathcal{O}_{Y\sslash{+} \times_{Y\sslash{\,0}} Y\sslash{-}} .
	\]
\end{cor}
\begin{proof}
	The equivalence of the wall-crossing functor is a direct application of Corollary~\ref{cor:ci}. We note that as the GIT quotients $ Y\sslash{\pm} $ are smooth, the category of perfect complexes is equivalent to the bounded derived category of coherent sheaves. 
	
	By Proposition~\ref{prop:fiber}, we know that the kernel is given by 
	\[
	\mathcal R^+ \otimes_{R_{(0)} } \mathcal R^-.
	\] We note that, as taking $ \mathbb G_m $-invariants is an exact functor, we have the isomorphism
	\[
		R_{(0)} \cong S_{(0)},
	\]  and hence,  we have the following chain of isomorphisms:
	\begin{align*}
	R \otimes_{R_{(0)} } R &\cong k[\mathbf x, \mathbf y, \mathbf x',\mathbf y', e]/(\mathbf x\mathbf y - \mathbf x' \mathbf y')\\
	&\cong k[\mathbf x, \mathbf y, \mathbf x',\mathbf y']/(\mathbf x\mathbf y - \mathbf x' \mathbf y', \sum_{a=1}^l x_ay_a) \\
	&\cong S \otimes_{S_{(0)}} S.
	\end{align*}
	 In the first line, we use the shorthand notation $ \mathbf x $ to denote the set of variables $ x_1,\cdots, x_l $ (similarly for $ \mathbf y $) and $ \mathbf x\mathbf y - \mathbf x' \mathbf y' $ to denote the set of functions $ x_i y_j - x'_iy'_j $, where $ i, j \in \{1,2,\cdots, l\} $. In order to get the last line, we merely note that the function $ \sum_{a=1}^l x'_ay'_a $ is zero, as $ x_ay_a = x'_ay'_a $. Finally, we restrict the above isomorphism to $ \X\sslash{+} \times \X\sslash{-} $ which is quasi-isomorphic as a dg-scheme to $ Y\sslash{+} \times Y\sslash{-} $, under the same chain of isomorphisms given above. This proves the statement about the wall-crossing kernel.
\end{proof}

\subsubsection{The wall-crossing kernel}\label{sec:Qpm}
On the level of the GIT quotient dg-schemes,  we know that the Fourier-Mukai kernel for the wall-crossing functor is $ Q|_{\X\sslash{+} \times \X\sslash{-}} $ up to a twist. However, after using the identification of the  dg-scheme $ \X\sslash{\pm} $ and the the (ordinary) scheme $ Y\sslash{\pm} $, it is not easy to identify the kernel for the wall-crossing functor $ \Phi^{\operatorname{wc}} :  \operatorname{Perf}(Y\sslash{+}) \longrightarrow \operatorname{Perf}(Y\sslash{-}) $ explicitly, \ie   as a complex of sheaves on $ Y\sslash{+} \times Y\sslash{-} $. 

In the easiest examples of flops, it is known that the kernel is the sheaf $ \mathcal{O}_{Y\sslash{+} \times_{Y\sslash{\,0}} Y\sslash{-}} $.  For example, this is true for Atiyah/standard flops and  Mukai flops (but not for  stratified Mukai flops\cite{Cautis}). We also showed an analogous result in our setting of dg-schemes in Proposition~\ref{prop:fiber}.

 However, we want to stress that in spite of Proposition~\ref{prop:fiber} and Corollary~\ref{cor:mf}, our construction does not always give the fiber product $ \mathcal{O}_{Y\sslash{+} \times_{Y\sslash{\,0}} Y\sslash{-}} $ as the kernel for the wall-crossing functor.   Below, we give  an example where the kernel for the wall-crossing equivalence is not the fiber product. 
 
 We also note that the space $ Y $ in the example below is a complete intersection in affine space (hence, fairly `nice'). The GIT quotients $ Y\sslash{+} $ and $ Y\sslash{-} $ are both isomorphic to $ \mathbb A^1 \cup_{\operatorname{pt_1}} \mathbb P^1 \cup_{\operatorname{pt_2}} \mathbb A^1 $. In particular, they are connected, but are neither smooth nor irreducible.

\begin{exa}\label{ex:Qnotasheaf}
	Consider the cdga $ R = k[x_1,x_2,y_1,y_2,e_1,e_2] $ with a $ \mathbb Z $-grading given by
	\[
		\deg x_1 = \deg x_2 = 1, \qquad \deg y_1 = \deg y_2 = -1 \qquad \deg e_1 = \deg e_2 = 0,
	\] and the differential acting non-trivially as
	\[
	d e_1 = x_1y_1, \qquad d e_2 = x_2 y_2.
	\] The associated dg-scheme $ \X = (\Spec k[x_1,x_2,y_1,y_2], R ) $ is quasi-isomorphic to  to the scheme \hbox{ $ Y = \Spec k[x_1,x_2,y_1,y_2]/(x_1y_1,x_2y_2) $} which is a complete intersection in $ \mathbb A^4 $. The GIT quotients are  then
	$$
	Y\sslash{+} = V(x_1y_1,x_2y_2) \subset \operatorname{tot}_{\mathbb P^1_{x_1:x_2}} \mathcal O(-1)^{\oplus 2},\qquad 	Y\sslash{-} = V(x_1y_1,x_2y_2) \subset \operatorname{tot}_{\mathbb P^1_{y_1:y_2}} \mathcal O(-1)^{\oplus 2}.
	$$ 
	
	 In this case, the object $ Q $ is the cdga
	\[
	Q = k[x_1,x_2,z_1,z_2, e_1,e_2, u], \qquad de_1 = ux_1z_1, \qquad de_2 = ux_2z_2,
	\] and we want to understand its restriction to  $ Y\sslash{+} \times Y\sslash{-} $. By abuse of notation, we will denote the restriction of $ Q $ to $ Y \times Y $ also by $ Q $. Let us consider an open chart $ U_{x_1} \times U_{y_2} $ of $ Y \times Y $  which is defined by inverting $ x_1 $ and $ y_2 $. Then, $ Q|_{U_{x_1} \times U_{y_2}} $ is the complex
	\[\begin{tikzcd}
	k[x_1,x_1^{-1},x_2,z_1,z_2,z_2^{-1},u]/(ux_1z_1) \arrow[r,"ux_2z_2"] &	k[x_1,x_1^{-1},x_2,z_1,z_2,z_2^{-1},u]/(ux_1z_1).
	\end{tikzcd}
	\]
	This complex has two homologies, given by 
	\[
	H^0( Q|_{U_{x_1} \times U_{y_2}}) = k[x_1,x_1^{-1},x_2,z_1,z_2,z_2^{-1},u]/(ux_1z_1, ux_2z_2),
	\]  and 
	\[
	H^{-1}( Q|_{U_{x_1} \times U_{y_2}})  = 	z_1 k[x_1,x_1^{-1},x_2,z_1,z_2,z_2^{-1},u]/(ux_1z_1) \cong k[x_1,x_1^{-1},x_2,z_1,z_2,z_2^{-1},u]/(u).
	\]

As the  $R \otimes R $-module structure of $Q$ comes from the map $ p \otimes s $, it is fairly straightforward to see that 
	\[
		H^{0}( Q|_{U_{x_1} \times U_{y_2}}) \cong \mathcal O_{U_{x_1} \times_{Y\sslash{\,0}} U_{y_2}}.
	\]  Consider the subvariety
	 $$    Z:=   V(y_1,y_2) \times V(x_1,x_2) \subset Y\sslash{+} \times Y\sslash{-}, \qquad Z \cong \mathbb P^1_{x_1,x_2} \times \mathbb P^1_{y_1,y_2} .$$ Then, we can show that 
	\[
		H^{-1}( Q|_{[U_{x_1} \times U_{y_2} / \mathbb G_m^2]}) \cong \mathcal O_{Z  \cap [U_{x_1} \times U_{y_2} / \mathbb G_m^2] }.
	\] By carrying out similar calculations on an open cover, we get the result that 
	\[
	H^0(Q|_{Y\sslash{+} \times  Y\sslash{-}}) \cong \mathcal O_{Y\sslash{+} \times_{Y\sslash{\,0}} Y\sslash{-}}, \qquad H^{-1}( Q|_{Y\sslash{+} \times Y\sslash{-}}) \cong \mathcal O_{Z}.
	\]
	In particular, we see that the kernel for the wall-crossing functor is not even a sheaf.
	\end{exa}

\begin{rem}
As $ Y $ is a  complete intersection in a smooth space, we can use the theory of LG models to study the above example as well. Let $T =  k[x_1,x_2,y_1,y_2,t_1,t_2]$ be the bigraded ring with weights $(1,0),(1,0),(-1,0),(-1,0),(0,1),(0,1)$.  The derived Kn\"orrer periodicity result of \cite{Isik} shows that the category of matrix factorizations $ \operatorname{MF}^{\mathbb G_m \times \mathbb G_m}(\Spec  T - V(x_1,x_2) , s) $, where $ s = t_1x_1y_1 + t_2x_2y_2$ is equivalent to the category $ D^b(Y\sslash{+}) $. Then,  the  wall-crossing equivalence result of \cite{BFKe} for Landau-Ginzburg models shows that 
\[
 \operatorname{MF}^{\mathbb G_m \times \mathbb G_m}(\Spec T - V(x_1,x_2) , s)  \cong  \operatorname{MF}^{\mathbb G_m \times \mathbb G_m}(\Spec T - V(y_1,y_2) , s).
 \]
  Finally, by applying  derived Kno\"errer periodicity  again, we can prove the equivalence studied in the above example.  The kernel provided through the above explanation has not been presented or studied elsewhere but should, in some sense, be Koszul dual to the one described in the example.

\end{rem}

\newcommand{\etalchar}[1]{$^{#1}$}

\end{document}